\documentclass[reqno,letterpaper,11pt]{article}

\usepackage[draft=false]{hyperref}

\usepackage{times}

\usepackage{amsmath}
\usepackage{amsfonts,amssymb,amsmath,latexsym,wasysym,mathrsfs,bbm,stmaryrd}
\DeclareMathAlphabet{\mathscrbf}{OMS}{mdugm}{b}{n}
\usepackage[all]{xy}
\usepackage[usenames]{color}
\usepackage{epsfig}
\usepackage{enumitem}

\usepackage{amsthm}

\usepackage{thmtools}
\usepackage{thm-restate}
\declaretheorem[numberwithin=section]{theorem}
\declaretheorem[sibling=theorem]{proposition}
\declaretheorem[sibling=theorem]{definition}
\declaretheorem[sibling=theorem]{corollary}
\declaretheorem[sibling=theorem]{lemma}

\declaretheorem[sibling=theorem,name=Proposition / Definition]{propdef}
\declaretheorem[sibling=theorem,style=remark]{remark}

\numberwithin{equation}{section} 

\def\R{\mathbb R}
\def\C{\mathbb C}

\newcommand{\diff}{\,\mathrm{d}}

\def\1{\mathbbm 1}

\def\H{\mathscr{H}}

\def\S{\mathscr{S}}

\def\H{\mathcal{H}}

\def\tensor{\otimes}
\def\supp{\mathrm{supp}\,}

\newcommand{\cross}{\mathrm{cr}}
\newcommand{\gap}{\mathrm{gap}}

\newcommand{\Tr}{\mathrm{Tr}}

\newcommand{\vertiii}[1]{{\left\vert\kern-0.25ex\left\vert\kern-0.25ex\left\vert #1 
    \right\vert\kern-0.25ex\right\vert\kern-0.25ex\right\vert}}

\renewcommand\emptyset{\varnothing}

\newcommand{\ip}[2]{\left\langle\,{#1}, {#2}\,\right\rangle} 

\begin{document}

\title{Segal-Bargmann transform:\\the $q$-deformation}
\author{Guillaume C\'ebron \\
Universit\'e Paul Sabatier\\
Institut de Math\'ematiques de Toulouse\\
118 Route de
Narbonne, 31062 Toulouse, France\\
\texttt{guillaume.cebron@math.univ-toulouse.fr}
\and
Ching-Wei Ho \\
Department of Mathematics \\
University of California, San Diego \\
La Jolla, CA 92093-0112 \\
\texttt{cwho@ucsd.edu}
}

\date{\today}

\maketitle

\begin{abstract}
We give identifications of the $q$-deformed Segal-Bargmann transform and define the Segal-Bargmann transform on mixed $q$-Gaussian variables. We prove that, when defined on the random matrix model of \'Sniady for the $q$-Gaussian variable, the classical Segal-Bargmann transform  converges to the $q$-deformed Segal-Bargmann transform in the large $N$ limit. We also show that the $q$-deformed Segal-Bargmann transform can be recovered as a limit of a mixture of classical and free Segal-Bargmann transform.
\end{abstract}

\tableofcontents
\newpage

\section{Introduction}

Let $H$ be a real finite-dimensional Hilbert space. Let $\gamma$ be the standard Gaussian measure on $H$,
whose density with respect to the Lebesgue measure at $h\in H$ is
$(2\pi)^{-d/2}\exp(-\|h\|_H^2/2)$. Let $\mu$ be the Gaussian measure on the complexification $H^{\mathbb{C}}=H+iH$ of $H$ 
whose density with respect to the Lebesgue measure at $h\in H^{\mathbb{C}}$ is
$\pi^{-d}\exp(-\|h\|_{H^{\mathbb{C}}}^2)$. 
For all $f\in L^2(H,\gamma)$, the map
$z\mapsto \int_{H}f(z-x)\diff \gamma (x),$
admits an analytic continuation $\mathscr{S}(f)$ to $H^{\mathbb{C}}$. Furthermore, the map $\mathscr{S}(f)$ is in the closed subspace of holomorphic functions of $L^2(H^{\mathbb{C}},\mu)$, hereafter denoted by $\mathcal{H}L^2(H^{\mathbb{C}},\mu)$. The resulting map\begin{equation}\mathscr{S}:L^2(H,\gamma)\to \mathcal{H}L^2(H^{\mathbb{C}},\mu)\label{SBeq}\end{equation} is known as the Segal-Bargmann transform, introduced by Segal \cite{Segal1963, Segal1978} and Bargmann \cite{Bargmann1961, Bargmann1962} in early 1960s.

\subsection{$q$-Deformed Segal-Bargmann Transform}
In \cite{Leeuwen1995}, Leeuwen and Massen considered a $q$-deformation of the Segal-Bargmann transform in the one-dimensional case. For all $0\leq q < 1$, the measure replacing the Gaussian measure is the $q$-Gaussian measure $\nu_q$ on $\mathbb{R}$, whose density with respect to the Lebesgue measure is
$$\nu_q(dx) = \1_{|x|\leq2/\sqrt{1-q}}\frac{1}{\pi}\sqrt{1-q}\sin\theta\prod_{n=1}^\infty (1-q^n)|1-q^ne^{2i\theta}|^2\diff x,$$
where $\theta\in [0,\pi]$ is such that $x=2\cos(\theta)/\sqrt{1-q}$. The $q$-deformation of the Segal-Bargmann transform is then defined through the kernel
$$\Gamma_q(x,z)= \prod_{k=0}^\infty \frac{1}{1-(1-q)q^kzx+(1-q)q^{2k}z^2},\ \  |x|\leq \frac{2}{\sqrt{1-q}},\ \ |z|<\frac{1}{\sqrt{1-q}}.$$
For all  function $f\in L^2(\mathbb{R},\nu_q)$, the function
$$\mathscr{S}_q(f):z\mapsto \int_{\mathbb{R}}f(x)\Gamma_q(x,z)\diff \nu_q (x)$$
is defined on the unit disk of radius $1/\sqrt{1-q}$ and the map $\mathscr{S}_q$ is in fact an isomorphism of Hilbert space between $L^2(\mathbb{R},\nu_q)$ and a reproducing kernel Hilbert space of analytic function on the unit disk of radius $1/\sqrt{1-q}$ which plays the role of the complexified version of $\nu_q$.

Let us remark that $\lim_{q\to 1}\nu_q(\diff x)=\exp(-x^2/2)/\sqrt{2\pi}\diff x$
and
$\lim_{q\to 1}\Gamma_q(x,z) = \exp(xz-z^2/2)$, which suggest to denote the standard normal distribution by $\nu_1$ and the classical Segal-Bargmann transform $\mathscr{S}$ on $L^2(\mathbb{R},\nu_1)$ by $\mathscr{S}_1$. The case $q=0$, studied in \cite{Biane1997b}, is of particular interest, since $\nu_0(\diff x)=\1_{|x|\leq 2}\sqrt{4-x^2}\frac{\diff x}{ 2\pi} $ is the well-known semicircular law and the so-called free Segal-Bargmann transform $\mathscr{S}_0$
maps isometrically $L^2(\mathbb{R},\nu_0)$ to the Hardy space of analytic functions on the unit disc. Although beyond the scope of this article, let us mention also the related work \cite{BlitvicKemp}, where Blitvi\'{c} and Kemp define a refinement of the $q$-deformed Segal-Bargmann transform.

\subsection{Matrix Approximations}
In~\cite{Biane1997b}, Biane proves that the free Segal-Bargmann transform $\mathscr{S}_0$ is the limit of the classical Segal-Bargmann transform on Hermitian matrices in the following sense: for all $N\geq 1$, let $\mathbb{M}_N$ be the space of complex matrices of size $N\times N$, let $\mathbb{H}_N$ be the subspace of Hermitian matrices $M=M^*$ of size $N\times N$, and let $\Tr$ denote the usual trace. Let $\gamma_N$ be the standard Gaussian measure on $\mathbb{H}_N$ for the norm $\|M\|^2=N\Tr(MM^*)$,
and $\mu_N$ be the standard Gaussian measure on $\mathbb{M}_N=\mathbb{H}_N+ i\mathbb{H}_N$ for the norm $\|M\|^2=2N\Re \Tr(M^2)$.
This way we can consider the Segal-Bargmann transform $\mathscr{S}:L^2(\mathbb{H}_N,\gamma_N)\to \mathcal{H}L^2(\mathbb{M}_N,\mu_N)$ defined by \eqref{SBeq}. Biane extends the transform $\mathscr{S}$ to act on $\mathbb{M}_N$-valued functions, by applying $\mathscr{S}$ entrywise. More precisely, endowing $\mathbb{M}_N$ with the norm $\|M\|^2_{\mathbb{M}_N}=\Tr(MM^*)/N$, he considers the Hilbert space tensor products $L^2(\mathbb{H}_N,\gamma_N;\mathbb{M}_N)=L^2(\mathbb{H}_N,\gamma_N)\otimes \mathbb{M}_N$ and $\mathcal{H}L^2(\mathbb{M}_N,\mu_N;\mathbb{M}_N)=\mathcal{H}L^2(\mathbb{M}_N,\mu_N)\otimes \mathbb{M}_N$, as well as the boosted Segal-Bargmann transform
$$\mathscrbf{S}_{N}=\mathscr{S}\otimes Id_{\mathbb{M}_N}:L^2(\mathbb{H}_N,\gamma_N;\mathbb{M}_N)\to \mathcal{H}L^2(\mathbb{M}_N,\mu_N;\mathbb{M}_N).$$
Each polynomial can be seen as an element of $L^2(\mathbb{H}_N,\gamma_N;\mathbb{M}_N)$ (or of $\mathcal{H}L^2(\mathbb{M}_N,\mu_N;\mathbb{M}_N)$) via the polynomial calculus, and Biane proved that, restricted to those polynomial functions, the Segal-Bargmann transform $\mathscrbf{S}_{N}$ converges to the free Segal-Bargmann transform $\mathscr{S}_0$ in the following sense: for all polynomial $P$,
$$\lim_{N\to \infty}\left\|\mathscrbf{S}_{N}(P)-\mathscr{S}_0(P)\right\|_{\mathcal{H}L^2(\mathbb{M}_N,\mu_N;\mathbb{M}_N)}=0.$$
One of the motivation of this article is to prove that the $q$-deformed Segal-Bargmann transform $\mathscr{S}_q$ can also be approximated by the classical one for $0< q \leq 1$. In the model of Biane, $L^2(\mathbb{H}_N,\gamma_N;\mathbb{M}_N)$ is an approximation of $L^2(\mathbb{R},\nu_0)$ in the sense that, for all polynomial $P$, $\|P\|_{L^2(\mathbb{R},\nu_0)}=\lim_{N\to \infty}\|P\|_{L^2(\mathbb{H}_N,\gamma_N;\mathbb{M}_N)}$. In the case of $0< q \leq 1$, we replace the previous model by a model of \'{S}niady introduced in \cite{Sniady2001} in order to approximate $L^2(\mathbb{R},\nu_q)$. Let us briefly describe this model.

Let $d\geq 0$. We endow $\mathbb{M}_{d}$ by the inner products, quotient if necessary,
$\langle A,B\rangle_1=\frac{1}{d}\Tr(AB^*)$ and 
$\langle A,B\rangle_0=\Tr(A)\Tr(B^*).
$ For all $S\subset \{1,\ldots,N\}$, we define the inner product $\langle \cdot,\cdot\rangle_S$ on $\mathbb{M}_{d^N}\simeq \bigotimes_{r=1}^N\mathbb{M}_{d}$ to be the inner product of the Hilbert space tensor product
$\bigotimes_{r=1}^N\left(\mathbb{M}_{d},\langle A,B\rangle_{\1_S(r)}\right)$. Let $\sigma=(\sigma_S)_{S\subset \{1,\ldots,N\}}$ be a family of real numbers indexed by all subsets of $\{1,\ldots,N\}$. It determines an averaged inner product
$\langle A,B\rangle_\sigma=\sum_{S\subset \{1,\ldots,n\}}\sigma_S^2\cdot \langle A,B\rangle_S$
on $\mathbb{M}_{d^N}$. Let $\gamma_{d^N}^\sigma$ be the Gaussian measure on $\mathbb{H}_{d^N}$ whose characteristic function is given by
$$\int_{\mathbb{H}_{d^N}} \exp(i\Tr(MX))\diff \gamma_{d^N}^\sigma(X)=\exp(-\|M\|^2_\sigma/2)$$
and $\mu_{d^N}^\sigma$ be the Gaussian measure on $\mathbb{M}_{d^N}$ whose characteristic function is given by
$$\int_{\mathbb{M}_{d^N}} \exp(i\Tr(MX))\diff \mu_{d^N}^\sigma(X)=\exp(-\Re \|M\|^2_\sigma/4).$$
Denoting by $\supp \gamma_{d^N}^\sigma$ the support of $\gamma_{d^N}^\sigma$, which is a linear subspace of $\mathbb{H}_{d^N}$, we have $\supp \mu_{d^N}^\sigma=\supp \gamma_{d^N}^\sigma+i\supp \gamma_{d^N}^\sigma$. The linear space $\supp \gamma_{d^N}^\sigma$ can be endowed with a unique inner product such that $\gamma_{d^N}^\sigma$ is the standard Gaussian measure on $\supp \gamma_{d^N}^\sigma$, and therefore the Segal-Bargmann transform $$\mathscr{S}:L^2(\mathbb{H}_{d^N},\gamma_{d^N}^\sigma)=L^2(\supp \gamma_{d^N}^\sigma,\gamma_{d^N}^\sigma)\to \mathcal{H}L^2(\supp \mu_{d^N}^\sigma,\mu_{d^N}^\sigma)=\mathcal{H}L^2(\mathbb{M}_{d^N},\mu_{d^N}^\sigma)$$ is well-defined as in~\eqref{SBeq}. Following the model of Biane, we consider the two following Hilbert space tensor products $L^2(\mathbb{H}_{d^N},\gamma_{d^N}^\sigma;\mathbb{M}_{d^N})=L^2(\mathbb{H}_{d^N},\gamma_{d^N}^\sigma)\otimes \mathbb{M}_{d^N}$ and $\mathcal{H}L^2(\mathbb{M}_{d^N},\mu_{d^N}^\sigma;\mathbb{M}_{d^N})=\mathcal{H}L^2(\mathbb{M}_{d^N},\mu_{d^N}^\sigma)\otimes \mathbb{M}_N$, where $\mathbb{M}_{d^N}$ is endowed with the norm $\|M\|^2_{\mathbb{M}_{d^N}}=\Tr(MM^*)/d^N$. Finally, we consider the boosted Segal-Bargmann transform
$$\mathscrbf{S}_{d^N}=\mathscr{S}\otimes Id_{\mathbb{M}_{d^N}}:L^2(\mathbb{H}_{d^N},\gamma_{d^N}^\sigma;\mathbb{M}_{d^N})\to \mathcal{H}L^2(\mathbb{M}_{d^N},\mu_{d^N}^\sigma;\mathbb{M}_{d^N}).$$
\begin{theorem}[see Theorem~\ref{theoremtwo}]\label{theoremone}Let $0\leq q \leq 1$. Under technical assumptions \ref{A}, \ref{B}, \ref{C} and \ref{D} on $\sigma$ (see Section~\ref{modelsniady}) which ensure that, for all polynomial $P$,$$\lim_{N\to \infty}\|P\|_{L^2(\mathbb{H}_{d^N},\gamma_{d^N}^\sigma;\mathbb{M}_{d^N})}=\|P\|_{L^2(\mathbb{R},\nu_q)},$$the Segal-Bargmann transform $\mathscrbf{S}_{d^N}$ converges to the $q$-deformed Segal-Bargmann transform $\mathscr{S}_q$ in the following sense: for all polynomial $P$,
$$\lim_{N\to \infty}\left\|\mathscrbf{S}_{d^N}(P)-\mathscr{S}_q(P)\right\|_{\mathcal{H}L^2(\mathbb{M}_{d^N},\mu_{d^N}^\sigma;\mathbb{M}_{d^N})}=0.$$
\end{theorem}

We are able to prove Theorem~\ref{theoremone} in the two parameter setting and in the multidimensional case.

\subsection{Two Parameter Case}A simple scaling of $\mathscr{S}:L^2(H,\gamma)\to \mathcal{H}L^2(H^{\mathbb{C}},\mu)$ gives us a unitary isomorphism $\mathscr{S}^{t}$ which depends on one parameter $t>0$. It is also possible to consider one scaling for the space $L^2(H,\gamma)$ and another scaling for the transform $\mathscr{S}$. It yields to the two-parameter Segal-Bargmann transform $\mathscr{S}^{s,t}$, where $s$ and $t$ are two parameters with $s>\frac{t}{2}>0$, which was defined by Driver and Hall in~\cite{DriverHall1999,Hall1999}. In this article, all the definition and results are considered in this two-parameter setting. In particular, we shall generalize the transform $\mathscr{S}_q$ of Leeuwen and Massen to a $q$-deformed Segal-Bargmann transform (with $-1<q<1$) given by
\begin{equation}
\mathscr{S}_q^{s,t}(f):z\mapsto \int_{\mathbb{R}} f(x)\Gamma_q^{s,t}(x,z)\nu_q^s(\diff x)\label{defSqst}
\end{equation}
where $\Gamma_q^{s,t}$ is a generating function and $\nu_q^s$ is scaled from $\nu_q$ so that it has variance $s$. With this formula, we are able to compute the range of the Segal-Bargmann transform, which is a reproducing kernel Hilbert space of analytic functions in an ellipse. It allows us to prove Theorem~\ref{theoremtwo}, which is a version of Theorem~\ref{theoremone} with two parameters $s$ and $t$.
\subsection{Multidimensional Case}
In~\cite{Biane1997b}, Biane extends the free Segal-Bargmann transform $\mathscr{S}_0$ to the multidimensional case, replacing $\mathbb{R}$ by an arbitrary real Hilbert space $H$. The space $L^2(H,\gamma)$ has to be replaced by a non-commutative generalization of a $L^2$-space. More precisely, in the classical case, $L^2(H,\gamma)$ can be viewed as the space of square-integrable random variables generated by the Gaussian field on $H$. If $-1\leq q \leq 1$, it is possible to define some $q$-deformations of Gaussian field over $H$ (see Section~\ref{qdeformationoftheSegalBargmanntransform}). The free Segal-Bargmann transform $\mathscr{S}_0$ acts on the space of square-integrable random variables generated by a $0$-deformed Gaussian field on $H$ (called semicircular system in~\cite{Biane1997b}).

In~\cite{Kemp2005}, Kemp generalizes Biane's results and defines a $q$-deformed Segal-Bargmann transform $\mathscr{S}_q$ acting on the space of square-integrable random variables generated by a $q$-deformed Gaussian field on $H$. In \cite{Ho2016}, the second author defined the two-parameter free Segal-Bargmann transform $\mathscr{S}_0^{s,t}$ acting on the space of square-integrable random variables generated by a $0$-deformed Gaussian field on $H$.
In this article, we will follow~\cite{Biane1997b,Ho2016,Kemp2005} and define the two-parameter $q$-deformed Segal-Bargmann transform $\mathscr{S}_q^{s,t}$ acting on the space of square-integrable random variables generated by a $q$-deformed Gaussian field on $H$.
Of course, if we consider $H=\mathbb{C}$, the $q$-Segal-Bargmann transform $\mathscr{S}_q^{s,t}$ is equivalent to the integral transform $\mathscr{S}_q^{s,t}$ already defined in~\eqref{defSqst}; that is to say the integral transform gives an explicit formula of the $q$-Segal-Bargmann transform in the one dimensional setting (see Corollary~\ref{frommultitoone}).

Theorem~\ref{theoremone} is true in the multidimensional case. Indeed, Theorem~\ref{theoremthree} shows that the two-parameter $q$-deformed Segal-Bargmann transform $\mathscr{S}_q^{s,t}$ acting on the space of square-integrable random variables generated by a $q$-deformed Gaussian field on $H$ can be approximated by the classical Segal-Bargmann transform.
\subsection{Mixture of $q$-Deformed Segal-Bargmann Transform}
In fact, it is possible to deform a Gaussian field over $\mathbb{R}^n$ in a much more complicated way, where a $q_{ii}$-deformed Gaussian random variable is considered for each direction of the canonical basis of $\mathbb{C}^n$, and where the correlation relation between two different variables is determined by some factors $q_{ij}$ ($q_{ij}=1$ yields the classical independence of random variables and $q_{ij}=0$ yields the free independence of random variables).

This deformation, first considered by Speicher in~\cite{Speicher1993}, is known as mixed $q$-Gaussian variables, and is uniquely determined by a symmetric matrix $Q=(q_{ij})_{1\leq i,j\leq n}$ with elements in $[-1,1]$. The case of the previous section corresponds to the case where all the elements of $Q$ are equal to a single $-1\leq q \leq 1$. It is also possible in this framework to define a $Q$-deformed Segal-Bargmann transform $\mathscr{S}_Q^{s,t}$, and restricted on the one-dimensional directions, $\mathscr{S}_Q^{s,t}$ yields to the already defined $\mathscr{S}_{q_{ii}}^{s,t}$. In particular, if all $q_{ii}$ are equal to $0$, $\mathscr{S}_Q^{s,t}$ can be seen as a noncommutative mixture of the classical Segal-Bargmann transform $\mathscr{S}^{s,t}$ (see Remark~\ref{remarkmix}).

In~\cite{Speicher1992}, Speicher proves the following central limit theorem: every $q$-deformed Gaussian random variable can be approximated by a normalised sum of mixed $q$-Gaussian variables for some appropriate choice of $Q$ with elements in $\{-1,1\}$. Similarly, M\l otkowski proves in~\cite{Mlotkowski2004} that the elements of $Q$ can be chosen in $\{0,1\}$ in the central limit theorem of Speicher.

Our last result, summed up in Theorem~\ref{theoremfour}, is the fact that the $q$-deformed Segal-Bargmann transform $\mathscr{S}_{q}^{s,t}$ can be approximated by a noncommutative mixture $\mathscr{S}_Q^{s,t}$ of the classical Segal-Bargmann transform applied on normalised sum of mixed $q$-Gaussian variables (see Remark~\ref{remarkmix}).

\subsection{Organization of the Paper}
A brief outline of the paper is as follows. In Section~\ref{Preliminaries}, we introduce the Segal-Bargmann transform $\mathscr{S}^{s,t}$, continue with a summary of the (mixed) $q$-random variables and end by a description of the random matrix model of \'{S}niady. In Section~\ref{Mainsection}, we introduce the two-parameter $q$-deformed Segal-Bargmann transform $\mathscr{S}^{s,t}_q$, and prove Theorem~\ref{theoremone}. In Section~\ref{Multisection}, we introduce the two-parameter $q$-deformed Segal-Bargmann transform in the multidimensional case, and prove Theorem~\ref{theoremthree}, the analogue of Theorem~\ref{theoremone} in this multidimensional setting. Finally, in Section~\ref{Mixsection}, we introduce the mixture of $q$-deformed Segal-Bargmann transform, and prove Theorem~\ref{theoremfour}.

\section{Preliminaries}\label{Preliminaries}
We begin by briefly introduce the already existing objects and results that will be useful for us: the two-parameter Segal-Bargmann transform, the $q$-deformation of the Gaussian measure, the $q$-deformation of independent Gaussian random variables and the model of random matrix of \'{S}niady which allows to approximate those $q$-deformed Gaussian random variables.
\subsection{Segal-Bargmann Transform}Let $H$ be a real finite-dimensional Hilbert space of dimension $d\geq 1$. For all $t>0$, we define $\gamma_t$ to be a Gaussian measure on $H$ whose density with respect to the Lebesgue measure at $x\in H$ is
$(2\pi t)^{-d/2}\exp(-\|x\|^2/2t).$ For all $r,s>0$, we define $\gamma_{r,s}$ to be a Gaussian measure on the complexification $H^{\mathbb{C}}=H+iH$ of $H$ whose density with respect to the Lebesgue measure at $x+iy\in H^{\mathbb{C}}$ is $(2\pi \sqrt{r s})^{-d}\exp(-\|x\|^2/2r-\|y\|^2/2s)$. In other words, identifying $H^{\mathbb{C}}=H+iH$ with $H\times H$, we have $\gamma_{r,s}=\gamma_{r}\otimes \gamma_{s}$: the parameters $r$ and $s$ define the respective scaling of the Gaussian measure on the real and the imaginary part of $H$.\label{fistpov}

In \cite{DriverHall1999}, Driver and Hall introduced a general version of the Segal-Bargmann transform which depends on two parameters $s$ and $t$. Let $s>t/2>0$. For all $f\in L^2(H,\gamma_s)$, the map
\begin{equation}z\mapsto \int_{H}f(z-x)\diff \gamma_t (x),\label{defconv}\end{equation}
has a unique analytic continuation $\mathscr{S}^{s,t}(f)$ to $H^{\mathbb{C}}$. Furthermore, the map $\mathscr{S}^{s,t}(f)$ is in the closed subspace of holomorphic functions of $L^2(H^{\mathbb{C}},\gamma_{s-t/2,t/2})$, denoted in the following by $\mathcal{H}L^2(H^{\mathbb{C}},\gamma_{s-t/2,t/2})$. The two parameter \emph{Segal-Bargmann transform} is the isomorphism of Hilbert space\begin{equation}\mathscr{S}^{s,t}:L^2(H,\gamma_s)\to \mathcal{H}L^2(H^{\mathbb{C}},\gamma_{s-t/2,t/2})\label{SBeqdeux}\end{equation}
The standard case considered by Segal and Bargmann corresponds to the case $s=t$, and the Segal-Bargmann $\mathscr{S}$ considered in the introduction corresponds to the case $s=t=1$.
\subsection{$q$-Gaussian Measure}
In this section, we will review some facts about $q$-Gaussian measures and $q$-Hermite polynomials. More discussions can be found in \cite{BozejkoKummererSpeicher1997,Szablowski2013}.

\begin{definition}Let $-1< q < 1$ and $t\geq 0$. The $q$-Gaussian measure $\nu_q$ of variance $1$ is defined to be
$$\nu_q(dx) = \1_{|x|\leq2/\sqrt{1-q}}\frac{1}{\pi}\sqrt{1-q}\sin\theta\prod_{n=1}^\infty (1-q^n)|1-q^ne^{2i\theta}|^2\;\diff x$$
where $\theta\in[0,\pi]$ is such that $x=2\cos\theta/\sqrt{1-q}$. The $q$-Gaussian measure $\nu_q^t$ of variance $t$ is given by $\nu_q^t(\diff x)=\nu_q(\diff x/\sqrt{t})$.
\end{definition}
Let $-1< q < 1$. For all integer $n$, set $[n]_q = 1+\cdots+q^{n-1}$. For all $t\neq 0$, the $q$-Hermite polynomials $H_n^{q,t}$ of parameter $t$ are defined by $H_0^{q,t}(x) = 1$, $H_1^q(x) = x$ and the recurrence relation
$$H_{n+1}^{q,t}(x) = x H_n^{q,t}(x)-t [n]_q H_{n-1}^{q,t}(x).$$ They form an orthogonal family with respect to $\nu_q$ with norm $[n]_q!t^n$.
Their generating function
$$\Gamma_q^t(x,z):=\sum_{k=0}^\infty \frac{z^k}{[k]_q!t^k}H_k^{q,t}(x) = \prod_{k=0}^\infty \frac{t}{t-(1-q)q^kzx+(1-q)q^{2k}z^2},$$
where $[n]_q! = \prod_{j=1}^n[j]_q$, converges whenever $|x|\leq \left|\frac{2\sqrt{t}}{\sqrt{1-q}}\right|$ and $|z|<\sqrt{\frac{t}{1-q}}$.

\subsection{$q$-Gaussian Variables and Wick Product}\label{GaussianvariablesandWickproduct}
\begin{definition}A non-commutative probability space $(\mathscr{A},\tau)$ is a unital $*$-algebra with a linear functional $\tau:\mathscr{A}\to \mathbb{C}$ such that $\tau[1_{\mathcal{A}}]=1$ and $\tau[A^*A]\geq 0$ for all $A\in \mathscr{A}$. The element of $\mathscr{A}$ are called random variables.

If $\mathcal{X}$ is a subset of $\mathscr{A}$, we denote by $L^2(\mathcal{X},\tau)$ the Hilbert space given by the completion of the (quotiented if necessary) space of the $*$-algebra generated by $\mathcal{X}$ with respect to the norm $\|A\|^2=\tau[A^*A]$, and by $\mathcal{H}L^2(\mathcal{X},\tau)$ the Hilbert space given by the completion of the (quotiented if necessary) space of the algebra generated by $\mathcal{X}$ with respect to the same norm.
\end{definition}
%
%

The following definition of $q$-Gaussian variables can be considered as a $q$-deformation of the Wick formula of Gaussian variables (the classical case corresponds to $q=1$). Let $\mathcal{P}_2(n)$ be the set of pairing of $\{1,\ldots,n\}$. Let $\pi$ be a pairing of $\{1,\ldots,n\}$. A quadruplet $1\leq i< j < k < l \leq n$ is called a crossing of $\pi$ if $\{i,k\}\in \pi$ and $\{j,l\}\in \pi$. The number of crossings of the pairing $\pi$ is denoted by $\cross(\pi)$.
\begin{definition}Let $-1\leq q \leq 1$. A set $\mathcal{X}$ of self-adjoint and centred non-commutative random variables in a non-commutative probability space $(\mathscr{A},\tau)$ is said to be \emph{jointly $q$-Gaussian} if, for all $X_1,\ldots,X_n\in \mathcal{X}$, we have
\begin{equation}
\tau[X_1\cdots X_n]=\sum_{\pi\in \mathcal{P}_2(n)}q^{\cross(\pi)}\prod_{\{i,j\}\in \pi}\tau[X_i X_j].\label{Wickq}\end{equation}
Two sets of jointly $q$-Gaussian variables $\mathcal{X}$ and $\mathcal{Y}$ are called \emph{$q$-independent} if and only if $\mathcal{X}\cup\mathcal{Y}$ is jointly $q$-Gaussian and the elements of $\mathcal{X}$ are orthogonal with the elements of $\mathcal{Y}$ in $L^2(\mathscr{A},\tau)$.

A set $\mathcal{Z}$ of non-commutative centred random variables in a non-commutative probability space $(\mathscr{A},\tau)$ is said to be \emph{jointly $q$-Gaussian} if $\{\Re Z,\Im Z:Z\in \mathcal{Z}\}$ is jointly $q$-Gaussian. Moreover, if $\tau[(\Re Z)^2]=s$ and $\tau[(\Im Z)^2]=t$, we say that $Z$ is a $(s,t)$-elliptic $q$-Gaussian variable.
\end{definition}

Let $\mathcal{X}$ be any set of jointly $q$-Gaussian variables (not necessarily self-adjoint). Then, by linearity, it follows that the linear span of $\{ X,X^*:X\in \mathcal{X}\}$ is also jointly $q$-Gaussian. If we take $X_1=\cdots=X_n=X=X^*$ in \eqref{Wickq} and such that $\tau[X^2]=1$, we obtain the formula for the moments of the $q$-Gaussian measure $\nu_q$ :
$$\tau[X^n]=\sum_{\pi\in \mathcal{P}_2(n)}q^{\cross(\pi)}=\int_{\mathbb{R}}x^n\nu_q(\diff x).$$
The $q$-Gaussian measure is called the distribution of the $q$-Gaussian variable $X$.
\begin{remark}A family of self-adjoint jointly $q$-Gaussian variables $(X_i)_{i\in I}$ is, up to isomorphism, a $q$-Gaussian process as defined in \cite{BozejkoKummererSpeicher1997}, with covariance $c:I\times I\to \mathbb{R}$ given by $c(i,j)=\tau[X_iX_j]$. We can view them as operators acting on a $q$-deformation of the Fock space over $H$ (see Section~\ref{mixedqgaussians}). In the literature, for example in \cite{BozejkoKummererSpeicher1997,BozejkoSpeicher1991,DonatiMartin2003,EffrosPopa2003,Kemp2005}, the $q$-Gaussian variables have often been considered in this particular representation. Since our work only involves the non-commutative distribution of the $q$-Gaussian variables, we found more convenient to forget about the representation of a $q$-Gaussian variables and define it via its non-commutative distribution. This non-commutative distribution is implicitly given in \cite[Proposition 2]{BozejkoSpeicher1991}, or alternatively in \cite[Corollary 2.1]{EffrosPopa2003}.
\end{remark}

\begin{definition}Let $-1\leq q \leq 1$. Let $n\geq 0$. The Wick product of $n$ jointly $q$-Gaussian variables $X_1,\ldots,X_n$, denoted by $X_1\diamond\cdots\diamond X_n$, is uniquely defined by the following recursion formula: the empty Wick product is $1$ and 
$$X_1\diamond\cdots\diamond X_n=X_1\cdot (X_2\diamond\cdots\diamond X_n)-\sum_{i=2}^nq^{i-1}\tau[X_1X_i](X_1\diamond\cdots\diamond\widehat{X_i}\diamond\cdots\diamond X_n)$$
where the hat means that we omit the corresponding element in the product.
\end{definition}

\begin{remark}The Wick product has been considered in \cite{BozejkoKummererSpeicher1997} and \cite{EffrosPopa2003} with different notation. Considering a set of self-adjoint jointly $q$-Gaussian variables $(X_i)_{i\in I}$ as a $q$-Gaussian process (as defined in \cite{BozejkoKummererSpeicher1997}) acting on the $q$-deformation of the Fock space over $L^2(X_i,\tau)$, the Wick product $X_1\diamond\cdots\diamond X_n$ coincides with the quantity denoted by $\Psi(X_1\otimes\cdots\otimes X_n)$ in \cite[Definition 2.5]{BozejkoKummererSpeicher1997} (they satisfy the same recursion formula thanks to \cite[Proof of Proposition 2.7]{BozejkoKummererSpeicher1997}). The Wick product $X_1\diamond\cdots\diamond X_n$ is denoted by ${:}X_1\cdots X_n{:}$ in \cite{EffrosPopa2003}.
\end{remark}

In \cite{EffrosPopa2003} is given an explicit formula for the Wick product of jointly $q$-Gaussian variables which are self-adjoint that we will present now. By linearity, the formula is also valid for non-necessarily self-adjoint variables. A Feynman diagram $\gamma$ on $\{1,\ldots,n\}$ is a partition of $\{1,\ldots,n\}$ into one- and two-element sets. The set of Feynman diagrams on $\{1,\ldots,n\}$ is denoted by $\mathcal{F}(n)$, and we have $\mathcal{P}_2(n)\subset \mathcal{F}(n)$. We extend naturally the notion of crossing to $\mathcal{F}(n)$: a quadruplet $1\leq i< j < k < l \leq n$ is called a crossing of $\gamma$ if $\{i,k\}\in \gamma$ and $\{j,l\}\in \gamma$. The number of crossings of a Feynman diagram $\gamma$ is denoted by $\cross(\gamma)$. Similarly, a triplet $1\leq i< j < k \leq n$ is called a gap of $\gamma$ if $\{i,k\}\in \gamma$ and $\{j\}\in \gamma$. The number of gaps of a Feynman diagram $\gamma$ is denoted by $\gap(\gamma)$. Finally, the number of pairings of a Feynman diagram $\gamma$ is denoted by $\sharp \gamma$.

\begin{theorem}[Theorem 3.1 of \cite{EffrosPopa2003}]The Wick product of jointly $q$-Gaussian variables $X_1,\ldots, X_n$ is given by
$$X_1\diamond\cdots\diamond X_n=\sum_{\gamma\in \mathcal{F}(n)}(-1)^{\sharp \gamma}q^{\gap(\gamma)-\cross(\gamma)}\prod_{\{a,b\}\in \gamma}\tau[X_{a}X_{b}]\overrightarrow{\prod_{\{c\}\in \gamma}}X_{c}.$$
\end{theorem}

In the following proposition, we sum up some properties of the Wick product which can be found in~\cite{BozejkoKummererSpeicher1997} and in~\cite{EffrosPopa2003} for self-adjoint jointly $q$-Gaussian variables. The general case follows by linearity.

\begin{proposition}\label{propWick}
\begin{enumerate}\item The Wick product is multilinear on the linear span of jointly {$q$-Gaussian} variables.
\item If $X_1,\ldots,X_{n}$ is jointly $q$-Gaussian,
$(X_1\diamond\cdots\diamond X_n)^*=X_n^*\diamond\cdots\diamond X_1^*.$
\item If $X_1,\ldots,X_{n+m}$ is jointly $q$-Gaussian (with $n,m\geq 0$),
$$\tau\Big[(X_1\diamond\cdots\diamond X_n)\cdot(X_{n+1}\diamond\cdots\diamond X_{n+m})\Big]=\delta_{n,m}\sum_{\pi\in \mathcal{P}_2(n,m)}q^{\cross(\pi)}\prod_{\{i,j\}\in \pi}\tau[X_iX_j].$$
\item If $\{X_i\}_{i\in I}$ is a set of jointly $q$-Gaussian variables, the set of Wick products$$\{X_{i(1)}\diamond\cdots\diamond X_{i(n)}\}_{n\geq 0,i(1),\ldots,i(n)\in I}$$is a spanning set of the algebra $\mathbb{C}\langle X_i:i\in I\rangle$ generated by $\{X_i\}_{i\in I}$.
\end{enumerate}
\end{proposition}
\subsection{Mixed $q$-Gaussian Variables}\label{mixedqgaussians}
Let $Q=(q_{ij})_{i,j\in I}$ be a symmetric matrix with elements in $[-1,1]$. We recall now the construction of the mixed $q$-Gaussian variables operators $X_i=c_i+c_i^*$, where $c_i$ satisfy the commutation relations of the form
\begin{equation}
\label{CommutationRelations}
c_i^*c_j = q_{ij} c_j c_i^* + \delta_{ij} 1.
\end{equation}
We consider a complex Hilbert space $K$ with an orthonormal basis $\{e_i\}_{i\in I}$, and the algebraic full Fock space 
$$ \mathcal F(K) =\C\Omega + \bigoplus_{n=1}^\infty (K)^{\tensor n}$$
where $\Omega$ is a unit vector called the vacuum. The set of permutations of $\{1,\ldots,n\}$ is denoted by $\mathfrak{S}_n$, and a pair $1\leq a< b < l \leq n$ is called an inversion of a permutation $\pi \in \mathfrak{S}_n$ if $\pi(a)\geq \pi(b)$. We define the Hermitian form $\langle \cdot, \cdot \rangle_{Q}$ to be the conjugate-linear extension of
\begin{align*}\langle \Omega, \Omega \rangle_{Q}&=1\\
\langle e_{i(1)}\otimes \cdots \otimes e_{i(k)},e_{j(1)}\otimes \cdots \otimes e_{j(\ell)} \rangle_{Q}&=\delta_{k\ell}\sum_{\substack{\pi \in \mathfrak{S}_k\\i=j\circ \pi}}\prod_{\{a,b\}\in inv(\pi)}q_{i(a)i(b)}.
\end{align*}
The $Q$-Fock space $\mathcal F_{Q}(K)$ is the completion of the quotient of $\mathcal F(K)$ by the kernel of $\langle \cdot, \cdot \rangle_{Q}$. For any $i\in I$, define the left creation operator $c_i$ on $\mathcal F_{Q}(K)$ to extend
\begin{align*}c_i(\Omega)&=e_i\\
c_i( e_{i(1)}\otimes \cdots \otimes e_{i(k)})&=e_i\otimes e_{i(1)}\otimes \cdots \otimes e_{i(k)}.
\end{align*}
The annihilation operator is its adjoint, which can be computed as
\begin{align}c_i^*(\Omega)&=0\\
c_i^*( e_{i(1)}\otimes \cdots \otimes e_{i(k)})&=\sum_{\ell=1}^k\delta_{ii(\ell)}q_{ii(1)}\cdots q_{ii(\ell-1)}\cdot e_{i(1)}\otimes \cdots \otimes e_{i(\ell -1)}\otimes e_{i(\ell +1)}\otimes \cdots \otimes e_{i(k)}.\label{defannihilation}
\end{align}
Finally, we define the mixed $q$-Gaussian variables $X_i$  to be $c_i+c_i^*$. We can compute explicitly the mixed moment of those variables with respect to the vector state $\tau[\cdot]=\langle \cdot \Omega, \Omega \rangle_Q$.
\begin{proposition}[Proof of Theorem 4.4 of \cite{BozejkoSpeicher1994}]We have
\begin{equation}
\tau[X_{i(1)}\cdots X_{i(n)}]=\sum_{\pi\in \mathcal{P}_2(n)}\prod_{\{a,b\}\in \cross(\pi)}q_{i(a)i(b)}\prod_{\{a,b\}\in \pi}\delta_{i(a)i(b)}.\label{Wickmixedq}\end{equation}
\end{proposition}
As a consequence, the distribution of the variable $X_i$ is the $q_{ii}$-Gaussian measure. Let us remark that if all the $q_{ij}$ are equal to a single $q$, the set $\{X_i\}_{i\in I}$ is jointly $q$-Gaussian. Finally, let us mention that it is also possible to define some Wick product for mixed $q$-Gaussian variables: see~\cite{JungeZeng2015,Krolak2000}.
%

\subsection{Random Matrix Model of \'Sniady}\label{modelsniady}
Let $d\geq 0$. We endow $\mathbb{M}_{d}$ by the inner products
$\langle A,B\rangle_1=\frac{1}{d}\Tr(AB^*)$ and 
$\langle A,B\rangle_0=\Tr(A)\Tr(B^*).
$ For all $S\subset \{1,\ldots,N\}$, we define the inner product $\langle \cdot,\cdot\rangle_S$ on $\mathbb{M}_{d^N}\simeq \bigotimes_{r=1}^N\mathbb{M}_{d}$ to be the inner product of the Hilbert space tensor product
$\bigotimes_{r=1}^N\left(\mathbb{M}_{d},\langle A,B\rangle_{\1_S(r)}\right)$. Let $\sigma=(\sigma_S)_{S\subset \{1,\ldots,N\}}$ be a family of real numbers indexed by all subsets of $\{1,\ldots,N\}$. We define the inner product on $\mathbb{M}_{d^N}$ given by
$$\langle A,B\rangle_\sigma=\sum_{S\subset \{1,\ldots,n\}}\sigma_S^2\cdot \langle A,B\rangle_S.$$
In order to be concrete, let us compute the inner product of elementary matrices. Setting $$T_{ij,kl}^{S}=\langle E_{j,i},E_{k,l}\rangle_{\1_S(r)}=\left\{\begin{matrix}
\frac{1}{d}\delta_{il}\delta_{jk} & \text{if }r\in S \\ 
\delta_{ij}\delta_{kl} & \text{if }r\notin S. \\ 
\end{matrix}\right.$$
and, for all $\mathbf{i}=(i_1,\ldots,i_N),\mathbf{j},\mathbf{k},\mathbf{l}\in \{1,\ldots,d\}^N$,
$$\mathbf{T}_{\mathbf{i}\mathbf{j},\mathbf{k}\mathbf{l}}^{S}=\langle E_{\mathbf{j},\mathbf{i}},E_{\mathbf{k},\mathbf{l}}\rangle_S=\prod_{r=1}^NT_{i_rj_r,k_rl_r}^{S},$$
we have, for all $\mathbf{i}=(i_1,\ldots,i_N),\mathbf{j},\mathbf{k},\mathbf{l}\in \{1,\ldots,d\}^N$,
\begin{equation}
\langle E_{\mathbf{j},\mathbf{i}},E_{\mathbf{k},\mathbf{l}}\rangle_\sigma=\sum_{S\subset \{1,\ldots,N\}}\sigma_S^2\mathbf{T}_{\mathbf{i}\mathbf{j},\mathbf{k}\mathbf{l}}^{S}.\label{covariance}
\end{equation}

\begin{theorem}[Theorem 1 of \cite{Sniady2001}]\label{thSniady}Let $\mathcal{X}=\{X_t\}_{t\in T}$ be a set of self-adjoint variables which are jointly  $q$-Gaussian.

For each $N\geq 0$, let $\sigma^{(N)}=(\sigma_S^{(N)})_{S\subset \{1,\ldots,N\}}$ be a family of real numbers, and let $\mathcal{X}^{(N)}=\{X_t^{(N)}\}_{t\in T}$ be the Gaussian stochastic process on $\mathbb{H}_{d^N}$ (indexed by $T$), uniquely defined by the following covariance: for all $M,N\in\mathbb{H}_{d^N}$ and all $s,t\in T$, one has
$$\mathbb{E}\left[\Tr(MX_t^{(N)})\Tr(NX_s^{(N)})\right]=\tau[X_tX_s]\langle M,N\rangle_{\sigma^{(N)}}.$$
In other words, the entries of the matrices in $\mathcal{X}^{(N)}$ are centered Gaussian variables with the following covariance: for all $\mathbf{i},\mathbf{j},\mathbf{k},\mathbf{l}\in \{1,\ldots,d\}^N$ and all $s,t\in T$, one has
$$\mathbb{E}\left[\Tr(E_{\mathbf{j},\mathbf{i}}X_t^{(N)})\Tr(E_{\mathbf{l},\mathbf{k}}X_s^{(N)})\right]=\mathbb{E}\left[\Tr(E_{\mathbf{j},\mathbf{i}}X_t^{(N)})\overline{\Tr(E_{\mathbf{k},\mathbf{l}}X_s^{(N)})}\right]=\tau[X_tX_s]\sum_{S\subset \{1,\ldots,N\}}(\sigma_S^{(N)})^2\mathbf{T}_{\mathbf{i}\mathbf{j},\mathbf{k}\mathbf{l}}^{S}.$$
Under the technical assumptions \ref{A}, \ref{B}, \ref{C} and \ref{D}, $\mathcal{X}^{(N)}$ converges to $\mathcal{X}$ in noncommutative distribution in the following sense: for all $t_1,\ldots,t_n$, we have
$$\lim_{N\to \infty}\mathbb{E}\left[\frac{1}{d^N}\Tr(X_{t_1}^{(N)}\cdots X_{t_n}^{(N)})\right]=\tau[X_{t_1}\cdots X_{t_n}]. $$
\end{theorem}

Before presenting the technical assumptions \ref{A}, \ref{B}, \ref{C} and \ref{D}, let us present two simple examples of family of real numbers $\sigma^{(N)}=(\sigma_S^{(N)})_{S\subset \{1,\ldots,N\}}$ (for $N\geq 0$) fulfilling all assumptions. Those examples are taken from~\cite[Proposition 1 and 2]{Sniady2001}: if $q$ can be written as $q=\exp(c^2/d^2-c^2)$ for a real number $c>0$, the sequence of functions defined by
$$(\sigma_S^{(N)})^2=\left(\frac{c}{\sqrt{N}}\right)^{|S|}\left(1-\frac{c}{\sqrt{N}}\right)^{N-|S|}$$
fulfils the assumptions of Theorem~\ref{thSniady}; if $q$ can be written as $q=\exp(c^2/d^2-c^2)$ for a real number $c>0$, the sequence of functions defined for $N$ sufficiently large by
$$(\sigma_S^{(N)})^2=\left\{\begin{array}{cl}
\frac{1}{\binom{N}{\lfloor c\sqrt{N}\rfloor}} & \text{if }|S|=\lfloor c\sqrt{N}\rfloor \\ 
0 & \text{otherwise}
\end{array}\right.$$
fulfils the assumptions of Theorem~\ref{thSniady}.

\begin{definition}
For each $N\geq 0$, let $\sigma^{(N)}=(\sigma_S^{(N)})_{S\subset \{1,\ldots,N\}}$ be a family of real numbers. The assumptions \ref{A}, \ref{B}, \ref{C} and \ref{D} are given as follow:
\begin{enumerate}[label=\textbf{H.\arabic*},ref=H.\arabic*]
\item \label{A}for each $N\in \mathbb{N}$,
$$ \sum_{S\subset \{1,\ldots,N\}}(\sigma_S^{(N)})^2=1,$$
\item we have\label{B}
$$\lim_{N\to \infty}\sum_{\substack{S_1,S_2,S_3\subset \{1,\ldots,N\}\\ S_1\cap S_2\cap S_3=\emptyset}}(\sigma_{S_1}^{(N)})^2(\sigma_{S_2}^{(N)})^2(\sigma_{S_3}^{(N)})^2=0,$$
\item there exists a sequence $(p_i)_{i\geq 0}$ of nonnegative real numbers such that $\sum_{i\geq 0}p_i=1$, $\sum_{i\geq 0}^\infty p_i/d^{2i}=q$, and such that, for any $k\in \mathbb{N}$ and any nonnegative integers numbers $(n_{ij})_{1\leq i<j\leq k}$, we have\label{C}
$$\lim_{N\to \infty}\sum_{\substack{S_1,\ldots,S_k\subset \{1,\ldots,N\}\\ |S_i\cap S_j|=n_{ij},\text{ for }1\leq i<j\leq k}}(\sigma_{S_1}^{(N)})^2(\sigma_{S_1}^{(N)})^2(\sigma_{S_1}^{(N)})^2=\prod_{1\leq i<j\leq k}p_{n_{ij}},$$
\item for each $k\in \mathbb{N}$,\label{D}
$$\lim_{N\to \infty}\sum_{S_1,\ldots,S_n\subset \{1,\ldots,N\}}\frac{(\sigma_{S_1}^{(N)})^2\cdots(\sigma_{S_k}^{(N)})^2}{N^{2|A_1\setminus (A_2\cup \cdots \cup A_n)|}}=0.$$
\end{enumerate}
\end{definition}
\section{The Two-Parameter $q$-Deformed Segal-Bargmann Transform}
\label{Mainsection}
In this section, we define the $q$-deformed Segal-Bargmann transform $\S_q^{s,t}$ with parameters $s>t/2>0$ and prove Theorem~\ref{theoremtwo}, which reduces to Theorem~\ref{theoremone} when $s=t=1$.
\subsection{An Integral Representation} The integral representation for the one-parameter and the two-parameter cases are similar. The two-parameter case is in fact a generalization of the one-parameter; we separate here simply to make the presentation of the computations clearer.
\paragraph{One-Parameter Case}
\begin{definition}Let $-1< q <1$ and $t\geq 0$. 
We define the $q$-deformed Segal-Bargmann transform $\S_q^t$ by
$$\S_q^t f(z) = \int f(x)\Gamma_q^t(x,z)\nu_q^t(\diff x).$$ 
\end{definition}
Observe that $\S_q^t H_n^t (z)=z^n$ and $\S_q^t$ is injective (by looking at the Fourier expansion of $L^2(\nu_q^t)$).
\begin{remark}
When $t=1$, the transform $\S_q^t$ coincides with the the transform $W$ from \cite{Leeuwen1995}. The method is different; while van Leeuwen and Maassen discovered the integral kernel by solving an eigenvalue equation \cite[Equation (8)]{Leeuwen1995}, we make use of the generating function directly to match the result from the Fock space. The method we present here will give us a two-parameter generalization in later sections.
\end{remark}

\begin{theorem}
\label{qisomorphism}
The transform $\S_q^t$ is a unitary isomorphism between $L^2(\nu_q^t)$ and the reproducing kernel Hilbert space $\H_q^t$ of analytic functions on the disk $B\left(0, \sqrt{\frac{t}{1-q}}\right)$ generated by the positive-definite sesqui-analytic kenel
$$K_q^t(z,\zeta) = \int \Gamma_q^t(x,z)\Gamma_q^t(x,\bar{\zeta})\nu_q^t(\diff x) = \sum_{k=0}^\infty \frac{1}{[k]_q!}\left(\frac{z\zeta}{t}\right)^k.$$
\end{theorem}
\begin{proof}
Let us denote $\Gamma_q^t(x,z)$ by $\Gamma_z(x)$ for each $z\in B\left(0, \sqrt{\frac{t}{1-q}}\right)$ and $x\in \mathbb{R}$. We also write $K_\zeta(z)=K_q^t(z,\zeta)$ as an analytic function on $B\left(0, \sqrt{\frac{t}{1-q}}\right)$. Observe that 
$$\S_q^t f(\zeta) = \ip{f}{\bar{\Gamma}_\zeta}_{L^2(\nu_q^t)}.$$
Define $\H_q^t \equiv \S_q^t(L^2(\nu_q^t))$ equipped with the inner product
$$\ip{F}{G}_{\H_q} := \ip{(\S_q^t)^{-1}F}{(\S_q^t)^{-1}G}_{L^2(\nu_q^t)}$$
which is well-defined since $\S_q^t$ is injective on $L^2(\nu_q^t)$. By construction $\S_q^t$ is a unitary isomorphism between $L^2(\nu_q^t)$ and $\H_q^t$.
Finally, we see that $K_\zeta(z) = \S_q^t \bar{\Gamma}_{\zeta}(z)$ and, for any $F\in \H_q^t$,
$$\ip{F}{K_\zeta}_{\H_q^t} = \ip{(\S_q^t)^{-1}F}{(\S_q^t)^{-1}K_\zeta}_{L^2(\nu_q^t)}  =  \ip{(\S_q^t)^{-1}F}{h_\zeta}_{L^2(\nu_q^t)} = \S_q^t((\S_q^t)^{-1}F)(\zeta) = F(\zeta)$$
which shows that $K_q^t$ is a reproducing kernel for $\H_q^t$.
\end{proof}

\begin{remark}
Since $\S_q^t$ coincides with $W$ from \cite{Leeuwen1995}, the reproducing kernel Hilbert space $\H_q^t$ actually is equal to the space $H^2(D_q, \mu_q)$ considered in \cite{Leeuwen1995}.
\end{remark}

\paragraph{Analytic Continuation of a Generating Function}
In this subsection, we study the analytic continuation on $y$ to the following generating function
\begin{align*}
\Lambda(r, x, y) = \sum_{k=0}^\infty h_k^{q}(x)h_k^{q}(y)\frac{r^n}{(q)_n} = \frac{(r^2)_\infty}{\prod_{k=0}^\infty (1-4rq^kxy+2r^2q^{2k}(-1+2x+2y)-4r^3q^{3k}xy+r^4q^{4k})}.
\end{align*}
where $x, y\in[-1,1]$, $0<|r|<1$ which is either real or purely imaginary, and $h_k^{q}(x) = H_k^{q,s}(\frac{x}{2}\sqrt{1-q})$.\\

This formula is known as the $q$-Mehler formula and has been studied analytically and combinatorially; see e.g. \cite[Theorem 1.10]{BozejkoKummererSpeicher1997} or \cite[Equation (24)]{Ngo2002}. By a standard theorem (see \cite[Theorem 15.4]{RudinRealComplex}), the analytic continuation on the parameter $y$ of $\Lambda$ is to solve, for a single $0<|r|<1$ and all $x\in [-1,1]$, what $y$ make $1-4rq^kxy+2r^2q^{2k}(-1+2x+2y)-4r^3q^{3k}xy+r^4q^{4k}=0$. The equation $$4r^2q^{2k}y^2-4tq^k x(1+r^2q^{2k})y+r^4q^{4k}+1+2r^2q^{2k}(2x^2-1)=0$$ 
has solution
$$y=\frac{1}{2}\left(\left(\frac{1}{rq^k}+rq^k\right)x \pm i \left(\frac{1}{rq^k}-rq^k\right) \sqrt{1-x^2}\right).$$
It follows that precisely when
\begin{equation}
\begin{split}
\label{AnaContSol}
y&=
\begin{cases}
\frac{1}{2}\left(\left(\frac{1}{|r||q|^k}+|r||q|^k\right)x \pm i \left(\frac{1}{|r||q|^k}-|r||q|^k\right) \sqrt{1-x^2}\right)\;\;\;\text{if }r\in\R\\
\frac{1}{2}\left(\pm  \left(\frac{1}{|r||q|^k}+|r||q|^k\right) \sqrt{1-x^2}\right)+\left(\frac{1}{|r||q|^k}-|r||q|^k\right)ix \;\;\;\text{if }r\in i\R\\
\end{cases}
\end{split}
\end{equation}
for some $x\in[-1,1]$, $\Lambda(r,x,y)$ has a zero for the particular $y$. Denote $\Omega_{k,r}$ the bounded component, which contains $0$, of the complement of the ellipse. Let
$$\varphi_1(u)=\frac{1}{|r||q|^u}+|r||q|^u$$
and 
$$\varphi_2(u)=\frac{1}{|r||q|^u}-|r||q|^u.$$
The derivative $\varphi_1'(u) = (-\log |q|)\left(\frac{1}{|r||q|^u}-|r||q|^u\right)>0$ implies $\varphi_1$ is increasing. Obviously $\varphi_2$ is increasing. Therefore $\Omega_{k,r}$ is increasing. Whence the $y$ parameter in $\Lambda(r,x,y)$ can be analytically continued to the ellipse $\Omega_{0,r}$. \\

\begin{proposition}
\label{AnalyticCont}
The generating function $\Lambda(r, x, z)$
can be analytically continued to $x\in[-1,1]$ and $z\in \Omega_{0,r}$ which is an ellipse with major axis $[-\frac{1}{2}(1/|r|+|r|), \frac{1}{2}(1/|r|+|r|)]$ and minor axis $i[-\frac{1}{2}(1/|r|-|r|), \frac{1}{2}(1/|r|-|r|)]$.
\end{proposition}

\paragraph{Two-Parameter Case}

We intend to define the integral $q$-Segal-Bargmann transform $\S_q^{s,t}$ by
$$\S_q^{s,t} f(z) = \int f(x)\Gamma_q^{s,t}(x,z)\nu_q^s(\diff x)$$
where
\begin{align*}
\Gamma_q^{s,t}(x,z) = &\sum_{k=0}^\infty \frac{H_k^{q,s-t}(z)}{s^k}\frac{H_k^{q,s}(x)}{[n]_q!}\\
=& \sum_{k=0}^\infty (s-t)^{k/2}s^{k/2}\frac{H_k^{q,s-t}(z)}{s^k}\frac{H_k^{q,s}(x)}{[n]_q!}\\
=& \sum_{k=0}^\infty \left(1-\frac{t}{s}\right)^{k/2}\frac{H_k^{q,s-t}(z)H_k^{q,s}(x)}{[n]_q!}.
\end{align*}
By \cite[Theorem 1.10]{BozejkoKummererSpeicher1997},  this series converges for $|x|, |z|\leq 2\sqrt{t}/\sqrt{1-q}$ for real $x, z$.

{\bf Case $s>t$:}\\

It is easy to see that
$$\Gamma_q^{s,t}(x,z) = \Lambda\left(\sqrt{1-\frac{t}{s}}, \frac{x\sqrt{1-q}}{2\sqrt{s}}, \frac{z\sqrt{1-q}}{2\sqrt{s-t}}\right).$$ 
By proposition \ref{AnalyticCont}, $\Gamma_q^{s,t}(x,z)$ is defined as an analytic function on the ellipse $E_{s,t}$ with major axis 
$$\frac{2\sqrt{s-t}}{\sqrt{1-q}}\left(\left(1-\frac{t}{s}\right)^{-\frac{1}{2}}+\left(1-\frac{t}{s}\right)^{\frac{1}{2}}\right)=\frac{2(2s-t)}{\sqrt{s}\sqrt{1-q}}$$
and minor axis 
$$\frac{2\sqrt{s-t}}{\sqrt{1-q}}\left(\left(1-\frac{t}{s}\right)^{-\frac{1}{2}}-\left(1-\frac{t}{s}\right)^{\frac{1}{2}}\right)=\frac{2t}{\sqrt{s}\sqrt{1-q}}.$$

{\bf Case $s<t$:}\\

Similarly,
$$\Gamma_q^{s,t}(x,z) = \Lambda\left(i\sqrt{\frac{t}{s}-1}, \frac{x\sqrt{1-q}}{2\sqrt{s}}, \frac{z\sqrt{1-q}}{2\sqrt{s-t}}\right).$$ 
By proposition \ref{AnalyticCont}, $\Gamma_q^{s,t}(x,z)$ is defined as an analytic function on the ellipse $E_{s,t}$ with major axis on the purely imaginary axis of length
$$\frac{2\sqrt{t-s}}{\sqrt{1-q}}\left(\left(\frac{t}{s}-1\right)^{-\frac{1}{2}}+\left(\frac{t}{s}-1\right)^{\frac{1}{2}}\right) = \frac{2(2s-t)}{\sqrt{s}\sqrt{1-q}}$$
and minor axis on the real axis of length
$$\frac{2\sqrt{t-s}}{\sqrt{1-q}}\left(\left(\frac{t}{s}-1\right)^{-\frac{1}{2}}-\left(\frac{t}{s}-1\right)^{\frac{1}{2}}\right) = \frac{2t}{\sqrt{s}\sqrt{1-q}}.$$
\begin{remark}
When $q=0$, the ellipse coincides with the ellipse where the Brown measure of an elliptic element is distributed; see \cite{BianeLehner2001}.
\end{remark}

\subsection{The Integral Transform}
\begin{definition}Let $-1< q <1$ and $s>t/2> 0$.\label{defSqstgen}
We define the $q$-deformed Segal-Bargmann transform $\S_q^{s,t}$ by
$$\S_q^{s,t} f(z) = \int f(x)\Gamma_q^{s,t}(x,z)\nu_q^s(\diff x)$$
for all $f\in L^2(\nu_q^s)$. $\S_q^{s,t}f$ is an analytic function on the ellipse $E_{s,t}$.
\end{definition}
Observe that $\S_q^{s,t} H_n^s (z)=H_n^{s-t} (z)$. The two-parameter analogue of Theorem~\ref{qisomorphism} holds:
\begin{theorem}
The transform $\S_q^{s,t}$ is a unitary isomorphism between $L^2(\nu_q^s)$ and the reproducing kernel Hilbert space $\H_q^{s,t}$ of analytic functions on the ellipse $E_{s,t}$ generated by the positive-definite sesqui-analytic kenel
$$K_q^{s,t}(z,\zeta) = \int \Gamma_q^{s,t}(x,z)\Gamma_q^{s,t}(x,\bar{\zeta})\nu_q^s(\diff x).$$
\end{theorem}

\subsection{Segal-Bargmann Transform and Conditional Expectation}
The goal of this section is to prove Corollary~\ref{Condexpforone}, showing that the $q$-deformed Segal-Bargmann transform can be
written as the action of a "$q$-deformed heat kernel". This result is already known for $q=0$, thanks to \cite[Theorem 3.1]{Cebron2013}.

Recall that the Wick product $X_1\diamond\cdots\diamond X_n$ is orthogonal in $L^2(\mathscr{A},\tau)$ to all products in $X_1,\ldots,X_n$ of degree strictly less than $n$. Since $X_1\cdots X_n-X_1\diamond\cdots\diamond X_n$ is in the span of the products in $X_1,\ldots,X_n$ of degree strictly less than $n$, $X_1\cdots X_n-X_1\diamond\cdots\diamond X_n$ can be seen as the orthogonal projection of $X_1\cdots X_n$ onto the span of the products in $X_1,\ldots,X_n$ of degree strictly less than $n$.
Because the Wick product can be seen as some orthogonal projection, the link with the conditional expectation is not surprising.

\begin{definition}Let $\mathcal{X}$ be a subset of a non-commutative space $(\mathscr{A},\tau)$. The conditional expectation $$\tau[\cdot |\mathcal{X}]:L^2(\mathscr{A},\tau)\to L^2(\mathcal{X},\tau)$$ is the orthogonal projection of $L^2(\mathscr{A},\tau)$ onto $L^2(\mathcal{X},\tau)$.
\end{definition}
\begin{remark}If $(\mathscr{A},\tau)$ is a $W^*$-probability space, that is to say a von Neumann algebra with an appropriate $\tau$, the conditional expectation $\tau[\cdot |\mathcal{X}]$ maps $\mathscr{A}$ into the von Neumann algebra $W^*(\mathcal{X})$ generated by $\mathcal{X}$.
\end{remark}
\begin{proposition}Let $\mathcal{X}$ and $\mathcal{Y}$ be two sets of jointly $q$-Gaussian variables which are $q$-independent. Let $X_1,\ldots,X_n \in \mathcal{X}\cup \mathcal{Y}$. We have
$$\tau\Big[X_1\diamond\cdots\diamond X_n\Big|\mathcal{X}\Big]=0$$
if one of the $X_i$s belongs to $\mathcal{Y}$, and $X_1\diamond\cdots\diamond X_n$ if all $X_i$s are in $\mathcal{X}$.\label{CondExpProp}
\end{proposition}

\begin{proof}If all $X_i$s are in $\mathcal{X}$, $X_1\diamond\cdots\diamond X_n$ is in $L^2(\mathcal{X},\tau)$ and the conditional expectation does not affect $X_1\diamond\cdots\diamond X_n$. If one of the $X_i$s belongs to $\mathcal{Y}$, it is sufficient to verify that $X_1\diamond\cdots\diamond X_n$ is orthogonal to $L^2(\mathcal{X},\tau)$, and it is an immediate consequence of the following fact: for all $X_{n+1},\ldots,X_{n+m} \in \mathcal{X}$,
$$\tau\Big[(X_1\diamond\cdots\diamond X_n)\cdot(X_{n+1}\diamond\cdots\diamond X_{n+m})^*\Big]=0.$$
Indeed, using of Proposition~\ref{propWick}, the computation of the trace always involves a factor $\tau[X_iX_j^*]$ between a $X_i\in \mathcal{Y}$ and a $X_j\in \mathcal{X}$, which vanishes.
\end{proof}

\begin{corollary}Let $\mathcal{X}=\{X_i\}_{i\in I}$, $\mathcal{Y}=\{Y_j\}_{j\in J}$ and $\mathcal{Z}=\{Z_j\}_{j\in J}$ be three sets of jointly $q$-Gaussian variables which are $q$-independent. The conditional expectations $\tau\left[\cdot|\mathcal{X}\cup\mathcal{Z}\right]:L^2(\mathscr{A},\tau)\to L^2(\mathcal{X}\cup\mathcal{Z},\tau)$ and $\tau\left[\cdot|\mathcal{X}\right]:L^2(\mathscr{A},\tau)\to L^2(\mathcal{X},\tau)$ coincide on $L^2(\mathcal{X}\cup\mathcal{Y},\tau)$.\label{CondExpCor}
\end{corollary}

\begin{proof}Thanks to Proposition~\ref{CondExpProp}, the two conditional expectations coincide on the Wick products of elements in $\mathcal{X}\cup\mathcal{Y}$ which is a dense subset of $L^2(\mathcal{X}\cup\mathcal{Y},\tau)$ (see Proposition~\ref{propWick}).
\end{proof}

\begin{corollary}Let $Z$ be a $(s,t)$-elliptic $q$-Gaussian variable in $(\mathscr{A},\tau)$. If $Y$ is a $(t,0)$-elliptic $q$-Gaussian variable which is $q$-independent from $Z$, we have, for all polynomial $P$,\label{Condexpforone}
$$\mathscr{S}_q^{s,t}P(Z)=\tau\left[P(Y+Z)|Z\right].$$
\end{corollary}

\begin{proof}It suffices to prove the theorem for the Hermite polynomials $\{H^{q,s}_n\}_{n\geq 0}$. Because $\mathscr{S}_q^{s,t}H^{q,s}_n=H^{q,s-t}_n$ for all $n\geq 0$, we need to prove that, for all $n\geq 0$,
$$\tau\left[H^{q,s}_n(Y+Z)|Z\right]=H^{q,s-t}_n(Z).$$
We compute
$$\tau[Z^2]=s-t\ \text{ and }\ \tau[(Y+Z)^2]=\tau[Z^2]+\tau[Y^2]=(s-t)+t=s,$$
and we deduce the following equalities by induction:
$$H^{q,s-t}_n(Z)=Z^{\diamond n}\ \text{ and }\  H^{q,s}_n(Y+Z)=(Y+Z)^{\diamond n}.$$
Let us conclude by the following computation where we use Proposition~\ref{CondExpProp}:
$$
 \tau\left[H^{q,s}_n(Y+Z)|Z\right]= \tau\left[(Y+Z)^{\diamond n}|Z\right]=Z^{\diamond n}=H^{q,s-t}_n(Z).
$$
\end{proof}

\subsection{Random Matrix Model}\label{RandomMatrixModel}
Let $\gamma_{d^N}^{\sigma,t}$ be the Gaussian measure on $\mathbb{H}_{d^N}$ whose characteristic function is given by
$$\int_{\mathbb{H}_{d^N}} \exp(i\Tr(MX))\diff \gamma_{d^N}^\sigma(X)=\exp(-t\|M\|^2_\sigma/2).$$
The measure $\gamma_{d^N}^\sigma$ is supported on the following vector subspace
$$K_\sigma=\{X\in \mathbb{H}_{d^N}:\Tr(MX)=0 \text{ for all } M\in \mathbb{H}_{d^N} \text{ such that } \|M\|^2_\sigma=0 \}.$$
In particular, if $\|\cdot\|^2_\sigma$ is not faithful, $\gamma_{d^N}^{\sigma,t}$ is not absolutely continuous with respect to the Lebesgue measure on $\mathbb{H}_{d^N}$. However, $\gamma_{d^N}^{\sigma,t}$  is absolutely continuous with respect to the Lebesgue measure on the vector space $K_\sigma$. More precisely, using the Riesz representation theorem, let us define the linear map $\phi:K_\sigma\to K_\sigma$ to be the unique linear map such that, for all $x,y\in K_\sigma$,
$\Tr(xy)=\langle \phi(x),y\rangle_\sigma.$
With respect to the Lebesgue measure on the vector space $K_\sigma$, the measure $\gamma_{d^N}^{\sigma,t}$ has density proportional to
$$\exp\left(-\frac{1}{2t}\Tr(x\phi(x))\right)=\exp\left(-\frac{1}{2t}\|\phi(x)\|_\sigma^2\right).$$
The quantity $\|\phi(x)\|_\sigma$ is known as the \emph{Mahalanobis distance} from $x$ to $0$, and it is the norm of $K_\sigma$ for which $\gamma_{d^N}^{\sigma,1}$ is the standard Gaussian measure.

We follows now Section~\ref{fistpov} in order to define the Segal-Bargmann transform $\mathscr{S}^{s,t}$ on $L^2(K_\sigma,\gamma_{d^N}^{\sigma,s})$. First, we consider the Gaussian measure $\mu_{d^N}^{\sigma,r,s}$ on $K_\sigma+i K_\sigma$ which is given by $\gamma_{d^N}^{\sigma,r}\otimes \gamma_{d^N}^{\sigma,s}$ when identifying $K_\sigma+i K_\sigma$ with $K_\sigma\times K_\sigma$. A short computation shows that $\mu_{d^N}^{\sigma,r,s}$ is the Gaussian measure on $\mathbb{M}_{d^N}$ whose characteristic function is given by
$$\int_{\mathbb{M}_{d^N}} \exp(i\Tr(MX^*))\diff \mu_{d^N}^\sigma(X)=\exp(-r\|\Re M\|^2_\sigma/2-s\|\Im M\|^2_\sigma/2).$$
The Segal-Bargmann transform $$\mathscr{S}^{s,t}:L^2(\mathbb{H}_{d^N},\gamma_{d^N}^{\sigma,s})=L^2(K_\sigma,\gamma_{d^N}^{\sigma,s})\to \mathcal{H}L^2(K_\sigma+iK_\sigma,\mu_{d^N}^{\sigma,s-t/2,t/2})=\mathcal{H}L^2(\mathbb{M}_{d^N},\mu_{d^N}^{\sigma,s-t/2,t/2})$$ is well-defined as in~\eqref{SBeqdeux}.

Following the model of Biane, we consider the two following Hilbert space tensor products$$L^2(\mathbb{H}_{d^N},\gamma_{d^N}^{\sigma,s};\mathbb{M}_{d^N})=L^2(\mathbb{H}_{d^N},\gamma_{d^N}^{\sigma,s})\otimes \mathbb{M}_{d^N}$$ and $$\mathcal{H}L^2(\mathbb{M}_{d^N},\mu_{d^N}^{\sigma,s-t/2,t/2};\mathbb{M}_{d^N})=\mathcal{H}L^2(\mathbb{M}_{d^N},\mu_{d^N}^{\sigma,s-t/2,t/2})\otimes \mathbb{M}_{d^N},$$ where $\mathbb{M}_{d^N}$ is endowed with the norm $\|M\|^2_{\mathbb{M}_{d^N}}=\Tr(MM^*)/d^N$. Finally, we consider the boosted Segal-Bargmann transform
$$\mathscrbf{S}_{d^N}^{s,t}=\mathscr{S}^{s,t}\otimes Id_{\mathbb{M}_{d^N}}:L^2(\mathbb{H}_{d^N},\gamma_{d^N}^\sigma;\mathbb{M}_{d^N})\to \mathcal{H}L^2(\mathbb{M}_{d^N},\mu_{d^N}^\sigma;\mathbb{M}_{d^N}).$$

\begin{theorem}\label{theoremtwo}Let $0\leq q < 1$. Assuming \eqref{A}, \eqref{B}, \eqref{C} and \eqref{D} on $\sigma$ ensures that the Segal-Bargmann transform $\mathscrbf{S}_{d^N}^{s,t}$ converges to the $q$-deformed Segal-Bargmann transform $\mathscr{S}_q^{s,t}$ in the following sense: for all polynomial $P$, we have
$$\lim_{N\to \infty}\left\|\mathscrbf{S}_{d^N}^{s,t}(P)-\mathscr{S}_q^{s,t}P\right\|_{\mathcal{H}L^2(\mathbb{M}_{d^N},\mu_{d^N}^{\sigma,s-t/2,t/2};\mathbb{M}_{d^N})}=0.$$
\end{theorem}

\begin{proof}Let us denote by $Q$ the polynomial $\mathscr{S}_q^{s,t}P$.

For all $z\in \mathbb{H}_{d^N}$, we have
\begin{align*}
(\mathscrbf{S}_{d^N}^{s,t}(P))(z)=\int_{\mathbb{H}_{d^N}}P(z-x)\diff \gamma^{\sigma,t}(x).
\end{align*}
Because both side are analytic in $z$, the equality is valid for all $z\in \mathbb{M}_{d^N}$. Thus we can compute
\begin{align*}
&\left\|\mathscrbf{S}_{d^N}^{s,t}(P)-Q\right\|_{\mathcal{H}L^2(\mathbb{M}_{d^N},\mu_{d^N}^{\sigma,s-t/2,t/2};\mathbb{M}_{d^N})}\\
&=\int_{\mathbb{M}_{d^N}}\int_{\mathbb{H}_{d^N}}\int_{\mathbb{H}_{d^N}}\left(P(z-x)-Q(z)\right)(P(z-y)-Q(z))^*\diff \gamma^{\sigma,t}(x)\diff \gamma^{\sigma,t}(y)\diff \mu^{\sigma,s-t/2,t/2}(z).
\end{align*}
Considering three independent random matrices $X^{(N)},Y^{(N)}$ and $Z^{(N)}$ of respective distribution $\gamma^{\sigma,t}$, $\gamma^{\sigma,t}$ and $\mu^{\sigma,s-t/2,t/2}$, we can rewrite
\begin{multline*}\left\|\mathscrbf{S}_{d^N}^{s,t}(P)-Q\right\|_{\mathcal{H}L^2(\mathbb{M}_{d^N},\mu_{d^N}^{\sigma,s-t/2,t/2};\mathbb{M}_{d^N})}\\=\mathbb{E}\left[(P(Z^{(N)}+X^{(N)})-Q(Z^{(N)}))(P(Z^{(N)}+Y^{(N)})-Q(Z^{(N)}))^*\right].
\end{multline*}
Let $X,Y$ be two $(t,0)$-elliptic $q$-Gaussian random variables and $Z$ be a $(s-t/2,t/2)$-elliptic $q$-Gaussian random variable such that $X,Y$ and $Z$ are $q$-independent. Remark that, for any random Hermitian matrix $X^{(N)}$ distributed according to $\gamma_{d^N}^{\sigma,t}$, for all $M,N\in\mathbb{H}_{d^N}$, one has
$$\mathbb{E}\left[\Tr(MX^{(N)})\Tr(NX^{(N)})\right]=t\langle M,N\rangle_{\sigma}.$$
Moreover, for any random matrix $Z$ distributed according to $\mu_{d^N}^{\sigma,s-t/2,t/2}$, $\Re Z$ and $\Im Z$ are two independent Hermitian random matrices distributed according to $\gamma_{d^N}^{\sigma,s-t/2}$ and $\gamma_{d^N}^{\sigma,t/2}$. Thus, we can apply Theorem~\ref{thSniady} which says that the Hermitian random matrices $X^{(N)},Y^{(N)},\Re Z^{(N)}$ and $\Im Z^{(N)}$ converge in noncommutative distribution to $X,Y,\Re Z$ and $\Im Z$. In particular, we have the following convergence:
\begin{multline*}\lim_{N\to \infty}\mathbb{E}\left[(P(Z^{(N)}+X^{(N)})-Q(Z^{(N)}))(P(Z^{(N)}+Y^{(N)})-Q(Z^{(N)}))^*\right]\\= \tau\left[(P(Z+X)-Q(Z))(P(Z+Y)-Q(Z))^*\right].\end{multline*}
From Corollary~\ref{CondExpProp} and Corollary~\ref{Condexpforone}, we know that $$Q(Z)=\mathscr{S}_q^{s,t}P(Z)=\tau[P(Z+X)|Z]=\tau[P(Z+X)|Z,Y].$$
Thus the limit $\tau\left[(P(Z+X)-Q(Z))(P(Z+Y)-Q(Z))^*\right]$ of $\left\|\mathscrbf{S}_{d^N}^{s,t}(P)-Q\right\|_{\mathcal{H}L^2(\mathbb{M}_{d^N},\mu_{d^N}^{\sigma,s-t/2,t/2};\mathbb{M}_{d^N})}$ vanishes:
\begin{align*}\tau\left[(P(Z+X)-Q(Z))(P(Z+Y)-Q(Z))^*\right]&=\tau\left[(P(Z+X)-\tau[P(Z+X)|Z,Y])(P(Z+Y)-Q(Z))^*\right]\\
&=\tau\left[(P(Z+X)-P(Z+X))(P(Z+Y)-Q(Z))^*\right]\\
&=0.
\end{align*}
\end{proof}

\section{Multidimensional $q$-Segal-Bargmann Transform}\label{Multisection}In this section, we will extend the definition of the $q$-Segal-Bargmann transform $\mathscr{S}_q^{s,t}$ to a multidimensional setting, and prove Theorem~\ref{theoremthree}, which says that Theorem~\ref{theoremtwo} is also true in this new setting. In order to understand the multidimensional case for $-1\leq q \leq 1$, we decide first to explain the infinite-dimensional case for the classical Segal-Bargmann transform.
\subsection{Classical Segal-Bargmann Transform in the Infinite-Dimensional Case}The content of this section is entirely expository. In Section \ref{sectionSB}, we shall define  a version of the Segal-Bargmann transform in a probabilistic framework which allows to consider infinite-dimensional Hilbert spaces. In Section \ref{SecSegalWick} and \ref{SecSegalCond}, we give two alternative descriptions of the Segal-Bargmann transform which are adapted to consider $q$-deformations.

\paragraph{In a probabilistic framework}In order to consider the $q$-deformation of this Segal-Bargmann transform, it is convenient to have a version of the $L^2$-spaces with more probabilistic flavor. Let $h\in H$. The continuous linear functional $\langle  \cdot ,h \rangle\in H^*$ can be considered as a random variable defined on the probability space $(H,\mathcal{B},\gamma_s)$ (where $\mathcal{B}$ is the Borel $\sigma$-field of $H$). Let us denote by $\mathbf{X}(h)$ the linear functional $x\mapsto \langle x,h\rangle$ defined on $H$ and by $\mathbf{Z}(h)$ the linear functional $z\mapsto \langle z,h\rangle$ defined on $H^{\mathbb{C}}$.
Because $H$ is finite-dimensional, the $\sigma$-field generated by the random variables $(\mathbf{X}(h))_{h\in H}$ is the Borel $\sigma$-field $\mathcal{B}$ of $H$. Denoting by $L^2(\mathbf{X})$ the random variables of $L^2(H,\mathcal{B},\gamma_s)$ which are measurable with respect to the $\sigma$-field generated by the random variables $(\mathbf{X}(h))_{h\in H}$, we have $L^2(\mathbf{X})=L^2(H,\mathcal{B},\gamma_s)$. Furthermore, it is well-known that the density in $L^2(H,\mathcal{B},\gamma_s)$ of the algebra $\mathbb{C}[\mathbf{X}(h):h\in H]$ of polynomial variable follows from H\"{o}lder inequality. Finally, the three following Hilbert spaces are identical:
$$L^2(\mathbf{X})=\overline{\mathbb{C}[\mathbf{X}(h):h\in H]}^{L^2(H,\mathcal{B},\gamma_s)}=L^2(H,\mathcal{B},\gamma_s).$$
In the same way, 
denoting by $\mathcal{H}L^2(\mathbf{Z})$ the completion of the algebra of random variables $\mathbb{C}[\mathbf{Z}(h):h\in H]$ in $L^2(H^{\mathbb{C}},\mathcal{B},\gamma_{s-t/2,t/2}
)$ we have the equality between the three following Hilbert spaces (where the first equality is a definition):
$$\mathcal{H}L^2(\mathbf{Z})=\overline{\mathbb{C}[\mathbf{Z}(h):h\in H]}^{L^2(H^{\mathbb{C}},\mathcal{B},\gamma_{s-t/2,t/2})}=\mathcal{H}L^2(H^{\mathbb{C}},\mathcal{B},\gamma_{s-t/2,t/2}).$$
The Segal-Bargmann map~\eqref{SBeqdeux} can now be seen as an isomorphism between two spaces of random variables $$\mathscr{S}^{s,t}:L^2(\mathbf{X})\to \mathcal{H}L^2(\mathbf{Z}).$$
From the definition~\ref{defconv}, the action of $\mathscr{S}^{s,t}$ on $\mathbb{C}[\mathbf{X}(h):h\in H]$ is easily described in the following way. The Hermite polynomials of parameter $s$ are defined by $H_0^s(x) = 1$, $H_1^s(x) = x$ and the recurrence relation
$xH_n^s(x) = H_{n+1}^s(x)+n s H_{n-1}^s(x).$ If $h_1,\ldots,h_k$ is an orthonormal family of $H$, the Hermite polynomials $H_{n_1}^s(\mathbf{X}(h_1))\cdots H_{n_k}^s(\mathbf{X}(h_k))$ form an orthonormal family of $L^2(\mathbf{X})$ and the action of $\mathscr{S}^{s,t}$ on this basis is\label{SBproba}
\begin{equation}
\mathscr{S}^{s,t}:H_{n_1}^s(\mathbf{X}(h_1))\cdots H_{n_k}^s(\mathbf{X}(h_k))\mapsto H_{n_1}^{s-t}(\mathbf{Z}(h_1))\cdots H_{n_k}^{s-t}(\mathbf{Z}(h_k)).\label{actiononHermite}
\end{equation}
The formula \eqref{actiononHermite} determines $\mathscr{S}^{s,t}$ on $\mathbb{C}[\mathbf{X}(h):h\in H]$ by linearity, and thus \eqref{actiononHermite} determines uniquely $\mathscr{S}^{s,t}$ on $L^2(\mathbf{X})$ by continuity.

\paragraph{In the infinite dimensional case}\label{sectionSB}
The first approach of Section~\ref{fistpov} can not extend directly to the infinite-dimensional setting because the Gaussian measures $\gamma_s$ do not make sense as measures on an infinite-dimensional Hilbert space. The dual point of view of Section~\ref{SBproba} allows to define the Segal-Bargmann transform on infinite-dimensional Hilbert spaces. Indeed, $\mathbf{X}$ and $\mathbf{Z}$ of last section are particular cases of what we will called Gaussian fields. One has just to replace the underlying probability space $(H,\mathcal{B},\gamma_s)$, which is not well-defined, by a sufficiently big one $(\Omega,\mathcal{F},\mathbb{P})$. In the following, the underlying probability space  $(\Omega,\mathcal{F},\mathbb{P})$ will be completely arbitrary, but in concrete cases, the measure of reference $\mathbb{P}$ is often supported on a space $\Omega$ bigger than $H$. For example, in \cite{DriverHall1999}, the measure of reference $\mathbb{P}$ is a Wiener measure on a Wiener space whose Cameron-Martin space is $H$.

Let us fix an underlying probability space $(\Omega,\mathcal{F},\mathbb{P})$ and call random variables the measurable functions on $\Omega$. For all real Hilbert space, a linear map $\mathbf{X}$ from $H$ to the space of real random variables is called a Gaussian field on $H$ if, for all $h\in H$, $\mathbf{X}(h)$ is centered Gaussian with variance $\mathbb{E}[|\mathbf{X}(h)|^2]=\|h\|^2$. For all $r,s\geq 0$, a linear map $\mathbf{Z}$ from $H$ to the space of complex random variables is called an $(r,s)$-elliptic Gaussian field if it has the same distribution as $\sqrt{r} \mathbf{Z}_1+i\sqrt{s} \mathbf{Z}_2$, where $\mathbf{Z}_1$ and  $\mathbf{Z}_2$ are two Gaussian fields on $H$ which are independent (in particular, an $(r,0)$-elliptic Gaussian field is real-valued and an $(0,s)$-elliptic Gaussian field is purely imaginary-valued). Let $r,s\geq 0$, and let $\mathbf{Z}$ be an $(r,s)$-elliptic Gaussian field. Following the last section, we define $\mathcal{H}L^2(\mathbf{Z})$ to be the completion of the algebra of random variables $\mathbb{C}[\mathbf{Z}(h):h\in H]$ in $L^2(\Omega,\mathcal{F},\mathbb{P})$. When $s=0$ or $r=0$, $\mathcal{H}L^2(\mathbf{Z})$ coincide with the random variables of $L^2(\Omega,\mathcal{F},\mathbb{P})$ which are measurable with respect to the $\sigma$-field generated by the random variables $(\mathbf{Z}(h))_{h\in H}$, and we will simply write $L^2(\mathbf{Z})$ instead of $\mathcal{H}L^2(\mathbf{Z})$.

Let $s> t/2\geq 0$. In Section~\ref{SBproba}, $\mathbf{X}$ was an $(s,0)$-elliptic Gaussian field and $\mathbf{Z}$ was an $(s-t/2,t/2)$-elliptic Gaussian field on a finite-dimensional Hilbert space $H$. Thanks to Section~\ref{SBproba}, we have the following proposition.
\begin{proposition}Let $H$ be a (possibly infinite-dimensional) Hilbert space $H$, $\mathbf{X}$ be an $(s,0)$-elliptic Gaussian field on $H$ and $\mathbf{Z}$ be an $(s-t/2,t/2)$-elliptic Gaussian field on $H$. The map given, for all orthonormal family $h_1,\ldots,h_k$ of $H$, by
\begin{equation}\mathscr{S}^{s,t}:H_{n_1}^s(\mathbf{X}(h_1))\cdots H_{n_k}^s(\mathbf{X}(h_k))\mapsto H_{n_1}^{s-t}(\mathbf{Z}(h_1))\cdots H_{n_k}^{s-t}(\mathbf{Z}(h_k)),\label{SegalHermite}\end{equation}
is a well-defined isometry from $\mathbb{C}[\mathbf{X}(h):h\in H]$ to $\mathbb{C}[\mathbf{Z}(h):h\in H]$ which extends uniquely to an isomorphism of Hilbert space $\mathscr{S}^{s,t}:L^2(\mathbf{X})\to \mathcal{H}L^2(\mathbf{Z})$, called in the following the (two-parameter) \emph{Segal-Bargmann transform}.
\end{proposition}
\paragraph{Segal-Bargmann transform and Wick products}In order to define $q$-deformation of the Segal-Bargmann transform, we give here a second description of the Gaussian fields and of the Segal-Bargmann transform defined in Section \ref{sectionSB}.

Let $\mathbf{X}$ be a Gaussian field on $H$.
The Wick product is the result of the Gram-Schmidt process for the basis of $L^2(\mathbf{X})$ given by monomials. More precisely, for all $n\geq 0$ and $h_1,\ldots, h_n\in H$, we define the Wick product $\mathbf{X}(h_1)\diamond\cdots \diamond\mathbf{X}(h_n)$ of $\mathbf{X}(h_1),\ldots,\mathbf{X}(h_n)$ as the unique element of $$\mathbf{X}(h_1)\cdots\mathbf{X}(h_n)+Span \{\mathbf{X}(k_1)\cdots\mathbf{X}(k_m):m< n, k_1,\ldots, k_m\in H\}$$
which is orthogonal to $Span \{\mathbf{X}(k_1)\cdots\mathbf{X}(k_m):m\leq n, k_1,\ldots, k_m\in H\}$, or equivalently, such that$$\mathbb{E}[(\mathbf{X}(h_1)\diamond\cdots \diamond\mathbf{X}(h_n))\cdot \mathbf{X}(k_1)\cdots\mathbf{X}(k_m)]=0$$for all $m< n, k_1,\ldots, k_m\in H$. In certain cases, the Wick product can be computed explicitly. For all $n\geq 0$, $m_1,\ldots,m_n\geq 1$ and $h_1,\ldots, h_n$ an orthonormal family of $H$, we have
$$\mathbf{X}(h_1)^{\diamond m_1}\diamond\cdots\diamond\mathbf{X}(h_n)^{\diamond m_n}\ =H_{m_1}^1(\mathbf{X}(h_1))\cdots H_{m_n}^1(\mathbf{X}(h_n)).$$

Let $\mathbf{Z}$ be a Gaussian $(s,t)$-elliptic system on $H$. In the same way, for all $n\geq 0$ and $h_1,\ldots, h_n\in H$, we define the Wick product $\mathbf{Z}(h_1)\diamond\cdots\diamond\mathbf{Z}(h_n)$ of $\mathbf{Z}(h_1),\ldots,\mathbf{Z}(h_n)$ as the unique element of $$\mathbf{Z}(h_1)\cdots\mathbf{X}(h_n)+Span \{\mathbf{Z}(k_1)\cdots\mathbf{Z}(k_m):m< n, k_1,\ldots, k_m\in H\}$$
which is orthogonal to $Span \{\mathbf{Z}(k_1)\cdots\mathbf{Z}(k_m):m\leq n, k_1,\ldots, k_m\in H\}$, or equivalently, such that$$\mathbb{E}[(\mathbf{Z}(h_1)\diamond\cdots\diamond\mathbf{Z}(h_n))\cdot \mathbf{Z}(k_1)\cdots\mathbf{Z}(k_m)]=0$$for all $m< n, k_1,\ldots, k_m\in H$. By multilinearity and the discussion above, for all $n\geq 0$, $m_1,\ldots,m_n\geq 1$ and $h_1,\ldots, h_n$ an orthonormal family of $H$, we have $
\mathbf{Z}(h_1)^{\diamond m_1}\diamond\cdots\diamond\mathbf{Z}(h_n)^{\diamond m_n}=H_{m_1}^{s-t}(\mathbf{Z}(h_1))\cdots H_{m_n}^{s-t}(\mathbf{Z}(h_n)).$

We are now able to give an alternative description of the Segal-Bargmann transform. Let $\mathbf{X}$ be a $(s,0)$-elliptic Gaussian system, and $\mathbf{Z}$ be a Gaussian $(s,t)$-elliptic system on $H$. From \eqref{SegalHermite}, we deduce that, for all orthonormal family $h_1,\ldots,h_k$ of $H$, we have
$$\mathscr{S}^{s,t}(\mathbf{X}(h_1)^{\diamond m_1}\diamond\cdots\diamond\mathbf{X}(h_n)^{\diamond m_n})=\mathbf{Z}(h_1)^{\diamond m_1}\diamond\cdots\diamond\mathbf{Z}(h_n)^{\diamond m_n}$$
which can be generalized by multilinearity to the following.
\begin{proposition}Let $\mathbf{X}$ be a $(s,0)$-elliptic Gaussian system, and $\mathbf{Z}$ be a Gaussian $(s,t)$-elliptic system on $H$. For all $n\geq 0$ and $h_1,\ldots, h_n\in H$, we have\label{SecSegalWick}
\begin{equation}\mathscr{S}^{s,t}(\mathbf{X}(h_1)\diamond\cdots\diamond\mathbf{X}(h_n))=\mathbf{Z}(h_1)\diamond\cdots\diamond\mathbf{Z}(h_n).\label{SegalWick}\end{equation}\end{proposition}
\paragraph{Segal-Bargmann transform and conditional expectations}
In the proof of Theorem~\ref{theoremthree}, we will need a third description of the Segal-Bargmann transform, which follows directly from the definition. Let $\mathbf{X}$ be a $(s,0)$-elliptic Gaussian system, and $\mathbf{Z}$ be a Gaussian $(s,t)$-elliptic system on $H$. If $\mathbf{Y}$ is a $(t,0)$-elliptic Gaussian system which is independent from $\mathbf{Z}$, we have, for all $P\in \C[ x_h:h\in H]$,
\begin{equation}\mathscr{S}^{s,t}(P(\mathbf{X}(h):h\in H))=\mathbb{E}\left[P(\mathbf{Z}(h)+\mathbf{Y}(h):h\in H)|\mathbf{Z}(h):h\in H\right].\end{equation}
Because the formula only involves finitely many variables $h$ for each $P\in \C[ x_h:h\in H]$, it is enough to prove the formula for finite-dimensional Hilbert spaces $H$. For convenience, we take the particular case of Section~\ref{SBproba}: $\mathbf{X}(h)$ is the linear functional $x\mapsto \langle x,h\rangle$ defined on $(H,\mathcal{B},\gamma_s)$ and $\mathbf{Z}(h)$ the linear functional $z\mapsto \langle z,h\rangle$ defined on $(H^{\mathbb{C}},\mathcal{B},\gamma_{s-t/2,t/2})$. Let $P\in \mathbb{C}[x_h:h\in H]$. For all $z\in H$,
\begin{align*}\mathscr{S}^{s,t}(P(\mathbf{X}(h):h\in H))(z)&=\int_HP(\mathbf{X}(h):h\in H)(z-x)\diff \gamma_t (x)\\
&=\int_HP(\langle z-x,h\rangle:h\in H)\diff \gamma_t (x)\\
\mathscr{S}^{s,t}(P(\mathbf{X}(h):h\in H))(z)&=\int_HP\Big((\mathbf{Z}(h))(z)-(\mathbf{X}(h))(x):h\in H\Big)\diff \gamma_t (x).\end{align*}
The last line is also valid for all $z\in H^\mathbb{C}$, since each side is analytic. We recognize the conditioning of two independent set of variables: by enlarging the underlying probability space, we assume that there exists a $(t,0)$-elliptic Gaussian system $\mathbf{Y}$ independent from $\mathbf{Z}$ and rewrite the last equality as follows.\label{SecSegalCond}
\begin{proposition}Let $\mathbf{X}$ be a $(s,0)$-elliptic Gaussian system, and $\mathbf{Z}$ be a Gaussian $(s,t)$-elliptic system on $H$. Let us assume that there exists a $(t,0)$-elliptic Gaussian system $\mathbf{Y}$ independent from $\mathbf{Z}$. For all $P\in \C[ x_h:h\in H]$, we have
\begin{equation}\mathscr{S}^{s,t}\Big(P(\mathbf{X}(h):h\in H)\Big)=\mathbb{E}\Big[P(\mathbf{Z}(h)+\mathbf{Y}(h):h\in H)\Big|\mathbf{Z}(h):h\in H\Big].\label{SegalCond}\end{equation}\end{proposition}
\subsection{The $q$-Deformation of the Segal-Bargmann Transform}
\label{qdeformationoftheSegalBargmanntransform}

\begin{definition}Let $-1\leq q \leq 1$. A $q$-Gaussian field $\mathbf{X}_q$ on $H$ is a linear map from $H$ to a non-commutative probability space $(\mathscr{A},\tau)$ which is an isometry for the $L^2$-norm and such that $(\mathbf{X}_q(h))_{h\in H}$ is jointly $q$-Gaussian.

A $(r,s)$-elliptic $q$-Gaussian field $\mathbf{Z}_q$ is a linear map from $H$ to a non-commutative probability space $(\mathscr{A},\tau)$ which can be decomposed as $\sqrt{r}\mathbf{X}_q+i\sqrt{s}\mathbf{Y}_q$, where $\mathbf{X}_q$ and $\mathbf{Y}_q$ are two $q$-Gaussian field which are $q$-independent. Elliptic $q$-Gaussian fields are $q$-independent if the previous decomposition holds simultaneously with $q$-Gaussian fields which are all $q$-independent.
\end{definition}
The following definition of the Segal-Bargmann transform in the infinite-dimensional case coincide with the classical Segal-Bargmann transform if $q=1$, with the definition of Kemp in~\cite{Kemp2005} if $s=t$, and with the definition of the second author in~\cite{Ho2016} if $q=0$.
\begin{propdef}
\label{SBtransformqFockspace}
Let $\mathbf{X}_q$ be a $q$-Gaussian $(s,0)$-elliptic system, and $\mathbf{Z}_q$ be a $q$-Gaussian $(s-t/2,t/2)$-elliptic system from $H$ to $\mathscr{A}$. 
The ($q$-deformed) \emph{Segal-Bargmann transform} $\mathscr{S}^{s,t}_q$ is the unique unitary isomorphism from $L^2(\mathbf{X}_q,\tau)$ to $\mathcal{H}L^2(\mathbf{Z}_q,\tau)$ such that, for all $h_1,\ldots, h_n\in H$,
\begin{equation}
\mathscr{S}^{s,t}_q\left(\mathbf{X}_q(h_1)\diamond\cdots\diamond\mathbf{X}_q(h_n)\right)=\mathbf{Z}_q(h_1)\diamond\cdots\diamond\mathbf{Z}_q(h_n).\label{actiononqHermite}
\end{equation}
\end{propdef}
We will see in Corollary~\ref{frommultitoone} that this transform is indeed a generalization of Definition~\ref{defSqstgen}.
\begin{proof}The unicity is clear. It remains to prove the existence and the unitarity. Let us first remark that, for all $h,k\in H$, we have
\begin{multline*}\tau[\mathbf{Z}_q(h)\mathbf{Z}_q(k)^*]=\tau[\Re\mathbf{Z}_q(h)\Re\mathbf{Z}_q(k)]+0+\tau[\Im\mathbf{Z}_q(h)\Im\mathbf{Z}_q(h)]\\=(s-t/2)\langle h, k\rangle_H+(t/2)\langle h, k\rangle_H=s\langle h, k\rangle_H=\tau[\mathbf{X}_q(h)\mathbf{X}_q(k)^*].\end{multline*}
Combined with Proposition~\ref{propWick}, it follows that, for all $h_1,\ldots, h_n\in H$ and $k_1,\ldots, k_m\in H$,
$$\langle \mathbf{X}_q(h_1)\diamond\cdots\diamond\mathbf{X}_q(h_n),\mathbf{X}_q(k_1)\diamond\cdots\diamond\mathbf{X}_q(k_m)\rangle_{L^2(\mathbf{X}_q,\tau)}=\langle \mathbf{Z}_q(h_1)\diamond\cdots\diamond\mathbf{Z}_q(h_n),\mathbf{Z}_q(k_1)\diamond\cdots\diamond\mathbf{Z}_q(k_m)\rangle_{\mathcal{H}L^2(\mathbf{Z}_q,\tau)}.$$
We deduce the existence of the unitary linear map $\mathscr{S}^{s,t}_q$ from $\mathbb{C}\langle \mathbf{X}_q(h):h\in H\rangle$ to $\mathbb{C}\langle \mathbf{Z}_q(h):h\in H\rangle$ given by \eqref{actiononqHermite}, and we extend this map to $\mathscr{S}^{s,t}_q:L^2(\mathbf{X}_q,\tau)\to \mathcal{H}L^2(\mathbf{Z}_q,\tau)$ by density.
\end{proof}
Here again, the $q$-deformed Segal-Bargmann transform can be
seen as the action of a "$q$-deformed heat kernel", a result which extends \cite[Theorem 3.1]{Cebron2013} to $-1\leq q \leq 1$.
\begin{theorem}\label{qSBCondexp}Let $\mathbf{X}_q$ be a $q$-Gaussian $(s,0)$-elliptic system, and $\mathbf{Z}_q$ be a $q$-Gaussian $(s,t)$-elliptic system from $H$ to $\mathscr{A}$. If $\mathbf{Y}_q$ is a $q$-Gaussian $(t,0)$-elliptic system which is $q$-independent from $\mathbf{Z}_q$, we have, for all noncommutative polynomial $P\in \C\langle x_h:h\in H\rangle$,
$$\mathscr{S}_q^{s,t}(P(\mathbf{X}_q(h):h\in H))=\tau\left[P(\mathbf{Y}_q(h)+\mathbf{Z}_q(h):h\in H)|\mathbf{Z}_q\right].$$
\end{theorem}

\begin{proof}For all $h_1,\ldots,h_n$, we define a polynomial $P_{h_1,\ldots,h_n}\in\C\langle x_h:h\in H\rangle$ by the following recursion formula: $P_{\emptyset}=1$ and 
$$P_{h_1,\ldots,h_n}=x_{h_1}\cdot P_{h_2,\ldots,h_n}-\sum_{i=2}^nq^{i-1}s\langle h_1, h_i\rangle_H\cdot P_{h_1,\ldots,\widehat{h}_i,\ldots, h_n}$$
where the hat means that we omit the corresponding element in the product. Since $\{P_{h_1,\ldots,h_n}\}_{n\geq 0,h_1,\ldots,h_n\in H}$ is a spanning set of $\C\langle x_h:h\in H\rangle$, it suffices to prove the theorem for those polynomials. Remark that, for all $h,k\in H$, $\tau[\mathbf{X}_q(h)\mathbf{X}_q(k)]=s\langle h, k\rangle_H$. Consequently, the variables $P_{h_1,\ldots,h_n}(\mathbf{X}_q(h):h\in H)$ and $\mathbf{X}_q(h_1)\diamond\cdots\diamond\mathbf{X}_q(h_n)$ satisfies the same recursion formula, and we have
$$P_{h_1,\ldots,h_n}(\mathbf{X}_q(h):h\in H)=\mathbf{X}_q(h_1)\diamond\cdots\diamond\mathbf{X}_q(h_n).$$
Similarly, we compute
$$\tau[(\mathbf{Y}_q+\mathbf{Z}_q)(h)\cdot(\mathbf{Y}_q+\mathbf{Z}_q)(k)]=\tau[\mathbf{Z}_q(h)\mathbf{Z}_q(k)]+\tau[\mathbf{Y}_q(h)\mathbf{Y}_q(h)]=(s-t)\langle h, k\rangle_H+t\langle h, k\rangle_H=s\langle h, k\rangle_H,$$
and we deduce the following equality by induction:
$$P_{h_1,\ldots,h_n}(\mathbf{Y}_q(h)+\mathbf{Z}_q(h):h\in H)=(\mathbf{Y}_q+\mathbf{Z}_q)(h_1)\diamond\cdots\diamond(\mathbf{Y}_q+\mathbf{Z}_q)(h_n).$$
Let us conclude by the following computation where we use Proposition~\ref{CondExpProp} to compute the conditional expectation:
\begin{align*}
 \tau\left[P_{h_1,\ldots,h_n}(\mathbf{Y}_q(h)+\mathbf{Z}_q(h):h\in H)|\mathbf{Z}_q\right]&= \tau\left[(\mathbf{Y}_q+\mathbf{Z}_q)(h_1)\diamond\cdots\diamond(\mathbf{Y}_q+\mathbf{Z}_q)(h_n)|\mathbf{Z}_q\right]\\
 &=\mathbf{Z}_q(h_1)\diamond\cdots\diamond\mathbf{Z}_q(h_n)\\
 &=\mathscr{S}^{s,t}_q\left(\mathbf{X}_q(h_1)\diamond\cdots\diamond\mathbf{X}_q(h_n)\right)\\
 &=P_{h_1,\ldots,h_n}(\mathbf{X}_q(h):h\in H).
\end{align*}

\end{proof}
Combining Theorem~\ref{qSBCondexp} with Corollary~\ref{Condexpforone}, we get the following result which relies Definition~\ref{defSqstgen} of $\mathscr{S}^{s,t}_qP$ and Definition~\ref{SBtransformqFockspace} of $\mathscr{S}^{s,t}_q\left(P(\mathbf{X}_q(h))\right)$ for one polynomial $P$.
\begin{corollary}Let $-1< q <1$. For a unit vector $h$ and a polynomial $P$, we have\label{frommultitoone}
$$\mathscr{S}^{s,t}_q\left(P(\mathbf{X}_q(h))\right)=\mathscr{S}^{s,t}_qP(\mathbf{Z}_q(h)).$$
\end{corollary}
%

\subsection{Large N Limit}
Let us construct a boosted version of the Gaussian $(s,0)$-elliptic system on a Hilbert space $H$. Let us consider the tensor product Hilbert space $H\otimes_{\mathbb{R}} \mathbb{H}_{d^N}$ of $H$ with $(\mathbb{H}_{d^N},\|\cdot\|_{\sigma})$. Let $\mathbf{X}:H\otimes_{\mathbb{R}} \mathbb{H}_{d^N}\to L^2(\mathbf{X})\subset L^2(\Omega,\mathcal{F},\mathbb{P})$ be a Gaussian $(s,0)$-elliptic system.
We define $$\mathbf{X}^{(N)}:H\to L^2(\mathbf{X})\otimes_{\mathbb{R}} \mathbb{H}_{d^N}\simeq L^2(\mathbf{X})\otimes_{\mathbb{C}} \mathbb{M}_{d^N}\subset L^2(\Omega,\mathbb{P}; \mathbb{M}_{d^N})$$ by duality as the unique linear map $\mathbf{X}^{(N)}$ from $H$ to the random variables with value in $\mathbb{M}_{d^N}$ such that
$\mathbf{X}(h\otimes M)=\Tr(M\mathbf{X}^{(N)}(h)).$
Each variable $\Tr(M\mathbf{X}^{(N)}(h))$ is Gaussian with the covariance given by
$$\mathbb{E}\left[\Tr(M\mathbf{X}^{(N)}(h))\Tr(N\mathbf{X}^{(N)}(k))\right]=s\langle h,k\rangle_H \langle M,N\rangle_{\sigma}.$$
In other words, if the norm of $h$ is $1$, the distribution of the random matrix $\mathbf{X}^{(N)}(h)$ is the Gaussian distribution $\gamma_{d^N}^{\sigma,s}$ of Section~\ref{RandomMatrixModel}, and if $k$ is another vector orthogonal to $h$, the random matrices $\mathbf{X}^{(N)}(h)$ and $\mathbf{X}^{(N)}(k)$ are independent.

Similarly, let $\mathbf{Z}:H\otimes_{\mathbb{R}} \mathbb{H}_{d^N}\to L^2(\mathbf{X})\subset L^2(\Omega,\mathcal{F},\mathbb{P})$ be a Gaussian $(s-t/2,t/2)$-elliptic system. We define $$\mathbf{Z}^{(N)}:H\to L^2(\mathbf{X})\otimes_{\mathbb{R}} \mathbb{H}_{d^N}\simeq L^2(\mathbf{X})\otimes_{\mathbb{C}} \mathbb{M}_{d^N} \subset L^2(\Omega,\mathbb{P}; \mathbb{M}_{d^N})$$ by duality as the unique linear map $\mathbf{Z}^{(N)}$ from $H$ to the random variables with value in $\mathbb{M}_{d^N}$ such that
$\mathbf{Z}(h\otimes M)=\Tr(M\mathbf{Z}^{(N)}(h)).$
The Segal-Bargmann transform $\mathscr{S}^{s,t}:L^2(\mathbf{X})\to \mathcal{H}L^2(\mathbf{Z})$ is well-defined as in~\eqref{sectionSB}. Finally, we consider the following boosted Segal-Bargmann transform
$$\mathscrbf{S}_{d^N}^{s,t}=\mathscr{S}^{s,t}\otimes Id_{\mathbb{M}_{d^N}}:L^2(\mathbf{X})\otimes \mathbb{M}_{d^N}\to \mathcal{H}L^2(\mathbf{Z})\otimes \mathbb{M}_{d^N}.$$

\begin{theorem}\label{theoremthree}Let $0\leq q \leq 1$. Assuming \eqref{A}, \eqref{B}, \eqref{C} and \eqref{D} on $\sigma$ ensures that the Segal-Bargmann transform $\mathscrbf{S}_{d^N}^{s,t}$ converges to the $q$-deformed Segal-Bargmann transform $\mathscr{S}_q^{s,t}:L^2(\mathbf{X}_q,\tau)\to\mathcal{H}L^2(\mathbf{Z}_q,\tau)$ in the following sense: for all polynomial $P$ and $Q\in \mathbb{C}\langle x_h:h\in H\rangle$ such that
$$\mathscr{S}_q^{s,t}(P(\mathbf{X}_q(h):h\in H))=Q(\mathbf{Z}_q(h):h\in H),$$the norm
$\left\|P(\mathbf{X}^{(N)}(h):h\in H)\right\|_{L^2(\mathbf{X})\otimes \mathbb{M}_{d^N}}=\left\|\mathscrbf{S}_{d^N}^{s,t}(P(\mathbf{X}^{(N)}(h):h\in H))\right\|_{\mathcal{H}L^2(\mathbf{Z})\otimes \mathbb{M}_{d^N}}$ converges, as $N$ tends to $\infty$, to $\left\|P(\mathbf{X}_q(h):h\in H)\right\|_{L^2(\mathbf{X}_q,\tau)}=\left\|Q(\mathbf{Z}_q(h):h\in H)\right\|_{\mathcal{H}L^2(\mathbf{Z}_q,\tau)}$  and
$$\lim_{N\to \infty}\left\|\mathscrbf{S}_{d^N}^{s,t}(P(\mathbf{X}^{(N)}(h):h\in H))-Q(\mathbf{Z}^{(N)}(h):h\in H)\right\|_{\mathcal{H}L^2(\mathbf{Z})\otimes \mathbb{M}_{d^N}}=0.$$
\end{theorem}

\begin{proof}Remark that, for all $M,N\in\mathbb{H}_{d^N}$, and all $h,k\in H$, we have
$$\mathbb{E}\left[\Tr(M\mathbf{X}^{(N)}(h))\Tr(N\mathbf{X}^{(N)}(k))\right]=\tau[\mathbf{X}_q(h)\mathbf{X}_q(k)]\langle M,N\rangle_{\sigma}.$$
We can apply Theorem~\ref{thSniady} which says that the random matrices $\{\mathbf X^{(N)}(h):h\in H\}$ converge in noncommutative distribution to $\{\mathbf X_q(h):h\in H\}$. In particular, we have the following convergences:
\begin{align*}\lim_{N\to \infty}\left\|P(\mathbf{X}^{(N)}(h):h\in H)\right\|_{L^2(\mathbf{X})\otimes \mathbb{M}_{d^N}}=\left\|P(\mathbf{X}_q(h):h\in H)\right\|_{L^2(\mathbf{X}_q,\tau)}
\end{align*}

The proof of the second limit uses the following lemma. Let $\mathbf{Y}:H\otimes_{\mathbb{R}} \mathbb{H}_{d^N}\to L^2(\mathbf{Y})\subset L^2(\Omega,\mathcal{F},\mathbb{P})$ be a Gaussian $(t,0)$-elliptic system independent from $(\mathbf Z^{(N)}(h))_{h\in H}$, and define $$\mathbf{Y}^{(N)}:H\to L^2(\mathbf{Y})\otimes_{\mathbb{R}} \mathbb{H}_{d^N}\simeq L^2(\mathbf{Y})\otimes_{\mathbb{C}} \mathbb{M}_{d^N}\subset L^2(\Omega,\mathbb{P}; \mathbb{M}_{d^N})$$ by duality as the unique linear map $\mathbf{Y}^{(N)}$ from $H$ to the random variables with value in $\mathbb{M}_{d^N}$ such that
$\mathbf{Y}(h\otimes M)=\Tr(M\mathbf{Y}^{(N)}(h)).$

\begin{lemma}\label{LemmaCondMat}For all $P\in \C[ x_h:h\in H]$, we have
\begin{equation}\mathscrbf{S}_{d^N}^{s,t}\Big(P(\mathbf{X}^{(N)}(h):h\in H)\Big)=\mathbb{E}\Big[P(\mathbf{Z}^{(N)}(h)+\mathbf{Y}^{(N)}(h):h\in H)\Big|\mathbf{Z}^{(N)}(h):h\in H\Big].\label{SegalCondMat}\end{equation}
\end{lemma}

\begin{proof}[Proof of Lemma~\ref{LemmaCondMat}]One can apply \eqref{SegalCond} for each coordinate of $P(\mathbf{X}^{(N)}(h):h\in H)$ in any basis of $\mathbb{M}_{d^N}$. Alternatively, one can reason as follows.

Because the formula only involves finitely many variables $h$ for each $P\in \C[ x_h:h\in H]$, it is enough to prove the formula for finite-dimensional Hilbert spaces $H$. Without loss of generality, we take the particular case of Section~\ref{SBproba}: $\mathbf{X}(h\otimes M)$ is the linear functional $x\otimes N\mapsto \langle x,h\rangle_H \Tr(NM)$ defined on $(H\otimes \mathbb{H}_{d^N},\mathcal{B},\gamma_s)$, $\mathbf{Y}(h\otimes M)$ is the linear functional $x\otimes N\mapsto \langle x,h\rangle_H \Tr(NM)$ defined on $(H\otimes \mathbb{H}_{d^N},\mathcal{B},\gamma_t)$ and $\mathbf{Z}(h\otimes M)$ the linear functional $z\otimes N\mapsto \langle z,h\rangle \Tr(N^*M)$ defined on $(H^{\mathbb{C}}\otimes \mathbb{M}_{d^N},\mathcal{B},\gamma_{s-t/2,t/2})$. We consider then the matrix-valued random variables $\mathbf{X}^{(N)}(h):x\otimes M\mapsto \langle x,h\rangle_H \cdot M$, $\mathbf{Y}^{(N)}(h):x\otimes M\mapsto \langle x,h\rangle_H\cdot  M$ and $\mathbf{Z}^{(N)}(h):z\otimes M\mapsto \langle z,h\rangle_{H^{\mathbb{C}}}\cdot M$.

Let $P\in \mathbb{C}[X_h:h\in H]$. We use here the definition \eqref{defconv} of the Segal-Bargmann transform $\mathscr{S}^{s,t}$, which is also valid for $\mathscrbf{S}^{s,t}$ by linearity: for all $z\otimes M \in H\otimes \mathbb{H}_{d^N} $,
\begin{align*}\mathscrbf{S}^{s,t}(P(\mathbf{X}^{(N)}(h):h\in H))(z\otimes M)&=\int_{H\otimes \mathbb{H}_{d^N}}P(\mathbf{X}^{(N)}(h):h\in H)(z\otimes M-x\otimes N)\diff \gamma_t (x\otimes N)\\
&=\int_{H\otimes \mathbb{H}_{d^N}}P(\langle z,h\rangle\cdot M-\langle x,h\rangle\cdot N:h\in H)\diff \gamma_t (x\otimes N)\end{align*}
$$=\int_{H\otimes \mathbb{H}_{d^N}}P\Big((\mathbf{Z}^{(N)}(h))(z\otimes M)-(\mathbf{Y}^{(N)}(h))(x\otimes N):h\in H\Big)\diff \gamma_t (x\otimes N).$$
The last line is also valid for all $z\otimes M \in H^{\mathbb{C}}\otimes \mathbb{M}_{d^N} $, since each side is analytic in $z\otimes M$. We recognize the wanted conditioning of Lemma~\ref{LemmaCondMat}.
\end{proof}

Let us consider an independent copy $\mathbf{W}^{(N)}$ of $\mathbf{Y}^{(N)}$. We consider also two $q$-Gaussian $(t,0)$-elliptic system $\mathbf{W}_q$ and $\mathbf{Y}_q$ which are $q$-independent from each others and from $\mathbf{Z}_q$.
Remark that, for all $M,N\in\mathbb{H}_{d^N}$, and all $h,k\in H$, we have
$$\mathbb{E}\left[\Tr(M\mathbf{U}^{(N)}(h))\Tr(N\mathbf{V}^{(N)}(k))\right]=\tau[\mathbf{U}_q(h)\mathbf{V}_q(k)]\langle M,N\rangle_{\sigma},$$
where the symbols $\mathbf{U}$ and $\mathbf{V}$ can be replaced by any from the symbols $\mathbf{W}_q,\mathbf{Y}_q,\Re \mathbf{Z}_q$ and $ \Im \mathbf{Z}_q$. Thus, we can apply Theorem~\ref{thSniady} which says that the random matrices $\{\mathbf W^{(N)}(h),\mathbf Y^{(N)}(h),\Re \mathbf Z^{(N)}(h),\Im \mathbf Z^{(N)}(h):h\in H\}$ converge in noncommutative distribution to $\{\mathbf W_q(h),\mathbf Y_q(h),\Re \mathbf Z_q(h),\Im \mathbf Z_q(h):h\in H\}$. In particular, we have the following convergence:
\begin{align*}
&\lim_{N\to \infty}\left\|\mathscrbf{S}_{d^N}^{s,t}(P(\mathbf{X}^{(N)}(h):h\in H))-Q(\mathbf{Z}^{(N)}(h):h\in H)\right\|_{\mathcal{H}L^2(\mathbf{Z})\otimes \mathbb{M}_{d^N}}\\
&=\lim_{N\to \infty}\left\|\mathbb{E}\left[P(\mathbf{Z}^{(N)}(h)+\mathbf{Y}^{(N)}(h):h\in H)\left|\mathbf{Z}\right. \right]-Q(\mathbf{Z}^{(N)}(h):h\in H)\right\|_{\mathcal{H}L^2(\mathbf{Z})\otimes \mathbb{M}_{d^N}}\\
&=\lim_{N\to \infty}\left\|\mathbb{E}\left[P(\mathbf{Z}^{(N)}(h)+\mathbf{Y}^{(N)}(h):h\in H)-Q(\mathbf{Z}^{(N)}(h):h\in H)|\mathbf{Z} \right]\right\|_{\mathcal{H}L^2(\mathbf{Z})\otimes \mathbb{M}_{d^N}}\\
&=\lim_{N\to \infty}\mathbb{E}\left[\frac{1}{d^N}\Tr \left(\left(P(\mathbf{Z}^{(N)}(h)+\mathbf{Y}^{(N)}(h):h\in H)-Q(\mathbf{Z}^{(N)}(h):h\in H)\right)\right.\right.\\
&\hspace{2cm}\cdot\left.\left.\left(P(\mathbf{Z}^{(N)}(h)+\mathbf{W}^{(N)}(h):h\in H)-Q(\mathbf{Z}^{(N)}(h):h\in H)\right)^*\right)\right]\\
&=\tau \left[\left(P(\mathbf{Z}_q(h)+\mathbf{Y}_q(h):h\in H)-Q(\mathbf{Z}_q(h):h\in H)\right)\right.\\
&\hspace{2cm}\cdot\left.\left(P(\mathbf{Z}_q(h)+\mathbf{W}_q(h):h\in H)-Q(\mathbf{Z}_q(h):h\in H)\right)^*\right]\\
&=\left\|\tau\left[P(\mathbf{Z}_q(h)+\mathbf{Y}_q(h):h\in H)-Q(\mathbf{Z}_q(h):h\in H)|\mathbf{Z}_q \right]\right\|_{\mathcal{H}L^2(\mathbf{Z}_q,\tau)}\\
&=\left\|\tau\left[P(\mathbf{Z}_q(h)+\mathbf{Y}_q(h):h\in H)|\mathbf{Z}_q \right]-Q(\mathbf{Z}_q(h):h\in H)\right\|_{\mathcal{H}L^2(\mathbf{Z}_q,\tau)}.
\end{align*}
The last quantity vanishes because Theorem~\ref{qSBCondexp} tells us that$$Q(\mathbf{Z}_q(h):h\in H)=\mathscr{S}_q^{s,t}(P(\mathbf{X}_q))=\tau\left[P(\mathbf{Z}_q(h)+\mathbf{Y}_q(h):h\in H)|\mathbf{Z}_q \right].$$
\end{proof}

\section{Mixture of Classical and Free Segal-Bargmann Transform}\label{Mixsection}
In this section, we shall define the Segal-Bargmann transform for a mixture of classical and free random variables and then we recover the $q$-Segal Bargmann trasnform in the limit.

\subsection{The Mixed $q$-Deformed Segal-Bargmann Transform}
Let $Q=(q_{ij})_{i,j\in I}$ be a symmetric matrix with elements in $[-1,1]$. We consider a complex Hilbert space $K$ with an orthonormal basis $\{e_i\}_{i\in I}$, the Fock space $\mathcal F_{Q}(K)$, and the set of mixed $q$-Gaussian variables $\{X_i\}_{i\in I }$ acting on $\mathcal F_{Q}(K)$ as defined in Section~\ref{mixedqgaussians}. The set $\{\sqrt{s}X_i\}_{i\in I }$ are the mixed $q$-Gaussian variables of variance $s$. Remark that the map $A\mapsto A(\Omega)$ extend to a unitary isomorphism from $L^2(\{\sqrt{s}X_i\}_{i\in I},\tau)$ to $\mathcal F_{Q}(K)$.

As in Section~\ref{GaussianvariablesandWickproduct}, we will define the mixed $q$-Gaussian $(s,t)$-elliptic variables as a set of variables $\{Z_i\}_{i\in I}$ indexed by $I$ such that $\{\Re Z_i,\Im Z_i\}_{i\in I}$ are a set of mixed $q$-Gaussian variables with prescribed variance. The first step is to replace the index set $I$ by the index set $I\times \{1,2\}$, and the matrix $Q$ by the matrix
$$\tilde{Q} = 
\begin{pmatrix}
Q & Q\\
Q & Q
\end{pmatrix}  =  Q\otimes \begin{pmatrix}
1 & 1\\
1 & 1
\end{pmatrix} .$$ We consider the complex Hilbert space $K^2$ with an orthonormal basis $\{e_{(i,1)},e_{(i,2)}\}_{i\in I}$. Considering the Fock space $\mathcal F_{\tilde{Q}}(K^2)$, we define the set of mixed $q$-Gaussian variables $\{X_{(i,1)},X_{(i,2)}\}_{i\in I }$ acting on $\mathcal F_{\tilde{Q}}(K^2)$ as defined in Section~\ref{mixedqgaussians}. Finally, we set the mixed $q$-Gaussian $(s-t/2,t/2)$-elliptic variables
$$Z_i=\sqrt{s-t/2}X_{(i,1)}+i\sqrt{t/2}X_{(i,2)}.$$
Remark that the map $A\mapsto A(\Omega)$ extend to a unitary isomorphism from $\mathcal{H}L^2(\{Z_i\}_{i\in I},\tau)$ to $\mathcal F_{\tilde{Q}}(K^2)$.

We are ready to define the mixed $q$-deformed Segal-Bargmann transform.

\begin{definition}
The mixed $q$-deformed Segal-Bargmann transform $\S_Q^{s,t}$ is the unitary isomorphism so that the following diagram commute:
\begin{displaymath}
    \xymatrix{
        \mathcal F_{Q}(K) \ar@{^{(}->}[rr]^{\delta_Q^{s,t}}  &&\mathcal F_{\tilde{Q}}(K^2)\\
        L^2(\{X_i\}_{i\in I},\tau)  \ar[rr]_{\S^{s,t}_Q} \ar[u]^{A\,\mapsto A\,\Omega}       &&\mathcal{H}L^2(\{Z_i\}_{i\in I},\tau) \ar[u]_{A\,\mapsto A\,\Omega}}
\end{displaymath}
where $\delta_Q^{s,t}$ is the Fock space extension of $\delta_Q^{s,t}(\sqrt{s}e_i) = \sqrt{s-t/2}e_{(i,1)}+i \sqrt{t/2}e_{(i,2)}$, meaning that
$$\delta_Q^{s,t}(h_1\tensor\ldots\tensor h_k) = \delta_Q^{s,t}(h_1)\tensor\ldots\tensor \delta_Q^{s,t}(h_k).$$
\end{definition}

For all $i\in I$, we have
$$H_n^{q_{ii},s}(\sqrt{s}X_i)\Omega=(\sqrt{s}e_{i})^{\otimes n}.$$
Indeed, the definition of the Hermite polynomials is adjusted with the definition \eqref{defannihilation} of the annihilation operator $c^*_{i(1)}$ in such a way that, by a direct induction, for all $n\geq 2$, we have
\begin{align*}H_n^{q_{ii},s}(\sqrt{s}X_i)\Omega=&\sqrt{s}(c_{i}+c^*_{i}) (\sqrt{s}e_{i})^{\otimes n-1}-s\sum_{k=2}^nq^{n-2}(\sqrt{s}e_{i})^{\otimes n-2}\\
=&\sqrt{s}c_{i}(\sqrt{s}e_{i})^{\otimes n-1}\\
=&(\sqrt{s}e_{i})^{\otimes n}.
\end{align*}
Similarly, we have
$$H_n^{q_{ii},s-t}(Z_i)\Omega=(\sqrt{s-t/2}e_{(i,1)}+i \sqrt{t/2}e_{(i,2)})^{\otimes n}.$$
We deduce the following result, which says that restricted on the different $L^2(X_i,\tau)$, the mixed $q$-deformed Segal-Bargmann transform $\S_Q^{s,t}$ coincides with the $q$-deformed Segal-Bargmann transform $\S_{q_{ii}}^{s,t}$.
\begin{proposition}For all $i\in I$ and all polynomial $P$, we have\label{onedimensionSB}
$$\S_Q^{s,t}(P(\sqrt{s}X_i))=(\S_{q_{ii}}^{s,t}P)(Z_i).$$
\end{proposition}
\subsection{The $q$-Segal-Bargmann Transform in the Limit}

Set $I=\mathbb{N}$. We choose $Q=(q_{ij}=q_{ji})_{i,j\in I}$ randomly in $\{-1,+1\}$ or in $\{0,+1\}$, as independent random variables, identically distributed, with $\mathbb{E}[q_{ij}]=q$. We consider the mixed $q$-Gaussian variables $\{\sqrt{s}X_i\}_{i\in I}$ of variance $s$ as defined in the previous section.

\begin{remark}Let us recall first that $q_{ii}=-1,0$ or $1$ means respectively that $X_i$ is a Bernoulli variable, a semicircular variable or a Gaussian variable. Secondly, $q_{ii}=0$ or $1$ means respectively that $X_i$ and $X_j$ are freely independent or classically independent.
\end{remark}

Let us consider the sum
$$X^{(n)}:=\frac{\sqrt{s}X_{1}+\ldots+\sqrt{s}X_{n}}{\sqrt{n}}.$$
These variables define an approximation of a $q$-Gaussian variable. Speicher's central limit theorem (\cite[Theorem 1]{Speicher1992}) makes this statement precise whenever $q_{ij}\in \{-1,+1\}$. If $q_{ij}\in \{0,+1\}$, it is not complicated (using for example the characterisation with cumulants of~\cite{SpeicherWysoczanski2016}) to prove that we fall in the framework of $\Lambda$-freeness of M{\l}otkowski. More precisely, if we define $\Lambda$ to be the set of $\{i,j\}$ such that $q_{ij}=1$ and $i\neq j$, the algebras generated by the different $X_{i}$ are $\Lambda$-free. We can use M{\l}otkowski's central limit theorem (\cite[Theorem 4]{Mlotkowski2004}) and we get the following result.

\begin{theorem}[Theorem 1 of \cite{Speicher1992} and Theorem 4 of \cite{Mlotkowski2004}]Almost surely, the variable $X^{(n)}$ converges to a $q$-Gaussian variable $X$ of variance $s$ in noncommutative distribution in the following sense: for all polynomial $P$, we have
$$\lim_{N\to \infty}\tau\left[P(X^{(n)})\right]=\tau[P(X)]. $$
\end{theorem}




We consider now the mixed $q$-Gaussian $(s-t/2,t/2)$-elliptic variables $\{Z_i\}_{i\in I}$, where the relations are governed by the matrix $\tilde{Q}$, and we set
$$Z^{(n)}:=\frac{Z_{1}+\ldots+Z_{n}}{\sqrt{n}}.$$
The entries of $\tilde{Q}$ are not any more independent but only block-independent. Nevertheless, as used in~\cite[Section 4.2]{Kemp2005}, a straightforward modification of Speicher's proof and of  M{\l}otkowski's proof generalizes the theorem to this case.

\begin{theorem}[Theorem 1 of \cite{Speicher1992} and Theorem 4 of \cite{Mlotkowski2004}]Almost surely, the variable $Z^{(n)}$ converges to a $q$-Gaussian $(s-t/2,t/2)$-elliptic variable $Z$ in noncommutative distribution in the following sense: for all polynomial $P$, we have
$$\lim_{N\to \infty}\tau\left[P(Z^{(n)})\right]=\tau[P(Z)]. $$
\end{theorem}
The following theorem says that the mixed $q$-Segal-Bargmann transform is also an approximation of the $q$-deformed case.
\begin{theorem}\label{theoremfour}Set $I=\mathbb{N}$. We choose $Q=(q_{ij}=q_{ji})_{i,j\in I}$ randomly in $\{-1,+1\}$ or in $\{0,+1\}$, as independent random variables, and identically distributed with $\mathbb{E}[q_{ij}]=q$ for $i>j$. We consider the mixed $q$-Gaussian variables $\{\sqrt{s}X_i\}_{i\in I}$ of variance $s$, the mixed $q$-Gaussian $(s-t/2,t/2)$-elliptic variables
$\{Z_i\}_{i\in I}$ and the mixed $q$-Segal-Bargmann transform $\S_Q^{s,t}:L^2(\{X_i\}_{i\in I},\tau)\to \mathcal{H}L^2(\{Z_i\}_{i\in I},\tau)$.

Almost surely, the Segal-Bargmann transform $\mathscr{S}_{Q}^{s,t}$ converges to the $q$-deformed Segal-Bargmann transform $\mathscr{S}_q^{s,t}$ in the following sense: considering the sums
$$X^{(n)}:=\frac{\sqrt{s}X_{1}+\ldots+\sqrt{s}X_{n}}{\sqrt{n}}\ \text{ and }\ Z^{(n)}:=\frac{Z_{1}+\ldots+Z_{n}}{\sqrt{n}},$$
for all polynomial $P$, we have $\displaystyle\lim_{n\to \infty}\left\|\mathscr{S}_{Q}^{s,t}\left(P(X^{(n)})\right)-\mathscr{S}_{q}^{s,t}P(Z^{(n)})\right\|_{\mathcal{H}L^2(\{Z_i\}_{i\in I},\tau)}=0.$
\end{theorem}

\begin{remark}\label{remarkmix}\begin{itemize}
\item We can choose $q_{ii}$ arbitrarily. For example, if we choose $q_{ii}=1$, Proposition~\ref{onedimensionSB} tells us that $\mathscr{S}_{Q}^{s,t}$ restricted to $L^2(\sqrt{s}X_{i},\tau)$ is the classical Segal-Bargmann transform $\mathscr{S}_{q_{ii}}^{s,t}$.

\item Now, assume that $q_{ij}\in \{0,+1\}$. We define $\Lambda$ to be the (random) set of $\{i,j\}$ such that $q_{ij}=1$ and $i\neq j$. The algebras generated by the different $X_{i}$ are $\Lambda$-free in the sense of \cite{Mlotkowski2004}. Decomposing $L^2(\sqrt{s}X_{i},\tau)=L^{\circ}_i\oplus \mathbb{C}$ (with $L^{\circ}_i$ composed of the operators $A$ such that $\tau[A]=0$), and $L^2(Z_{i},\tau)=\mathcal{H}L^{\circ}_i\oplus \mathbb{C}$ decomposed similarly, we have the $\Lambda$-free products observed in \cite{Mlotkowski2004}:
$$L^2(\{X_i\}_{i\in I},\tau)=\bigoplus_{(i(1),\ldots,i(m))\in S(I,\Lambda)}L^{\circ}_{i(1)}\otimes \ldots \otimes L^{\circ}_{i(m)},$$
$$\mathcal{H}L^2(\{Z_i\}_{i\in I},\tau)=\bigoplus_{(i(1),\ldots,i(m))\in S(I,\Lambda)}\mathcal{H}L^{\circ}_{i(1)}\otimes \ldots \otimes \mathcal{H}L^{\circ}_{i(m)},$$
where $S(I, \Lambda)$ is the set of reduced words over $I$ modulo the relations  $(\ldots,i,j,\ldots)\simeq (\ldots,j,i,\ldots )$ if $\{i,j\}\in \Lambda$ and $(\ldots,i,i,\ldots)\simeq (\ldots,i,\ldots)$, or, more specifically, a set of representatives of minimal length.
Finally, $\mathscr{S}_{Q}^{s,t}:L^2(\{X_i\}_{i\in I},\tau)\to\mathcal{H}L^2(\{Z_i\}_{i\in I},\tau)$ can be decomposed as
$$\bigoplus_{(i(1),\ldots,i(m))\in S(I,\Lambda)}\mathscr{S}_{q_{i(1)i(1)}}^{s,t}\otimes \ldots \otimes \mathscr{S}_{q_{i(n)i(n)}}^{s,t},$$
or as a $\Lambda$-free product of classical Segal-Bargmann transform whenever $q_{ii}=1$ for all $i\in I$.
\end{itemize}\end{remark}
\begin{proof}We consider the index set $I\times \{1,2,3,4\}$, and the matrix
$$R = 
\begin{pmatrix}
Q & Q& Q&Q\\
Q & Q& Q&Q\\
Q & Q& Q&Q\\
Q & Q& Q&Q
\end{pmatrix}  =Q\otimes  \begin{pmatrix}
1 & 1&1&1\\
1 & 1&1&1\\
1 & 1&1&1\\
1 & 1&1&1
\end{pmatrix}.$$ We consider the complex Hilbert space $K^4\supset K^2$ with an orthonormal basis $\{e_{(i,1)},e_{(i,2)},e_{(i,3)},e_{(i,4)}\}_{i\in I}$ and the Fock space $\mathcal F_{R}(K^4)$. We have the canonical inclusion $\mathcal F_{\tilde{Q}}(K^2)\subset \mathcal F_{R}(K^4)$ given by the natural extension of  $K^2\subset K^4$.
We define the set of mixed $q$-Gaussian variables $\{X_{(i,1)},X_{(i,2)},X_{(i,3)},X_{(i,4)}\}_{i\in I }$ acting on $\mathcal F_{\tilde{Q}}(K^4)$ as defined in Section~\ref{mixedqgaussians}, which is an extension of the already defined action of $\{X_{(i,1)},X_{(i,2)}\}_{i\in I }$ on $\mathcal F_{\tilde{Q}}(K^2)$. The action of the mixed $q$-Gaussian $(s-t/2,t/2)$-elliptic variables $Z_i$ extends to $\mathcal F_{R}(K^4)$ by
$$Z_i=\sqrt{s-t/2}X_{(i,1)}+i\sqrt{t/2}X_{(i,2)},$$
and the action of $Z^{(n)}$ extends to $\mathcal F_{R}(K^4)$ by
$$Z^{(n)}:=\frac{Z_{1}+\ldots+Z_{n}}{\sqrt{n}}.$$
Finally, we define also
$$Y_i=\sqrt{t}X_{(i,3)},\ Y^{(n)}:=\frac{Y_{1}+\ldots+Y_{n}}{\sqrt{n}},$$
and
$$W_i=\sqrt{t}X_{(i,4)},\ W^{(n)}:=\frac{W_{1}+\ldots+W_{n}}{\sqrt{n}}.$$

\begin{lemma}For all polynomial $Q$, we have
\begin{align*}&\left\|\mathscr{S}_{Q}^{s,t}\left(P(X^{(n)})\right)-Q(Z^{(n)})\right\|_{\mathcal{H}L^2(\{Z_i\}_{i\in I},\tau)}\\
&=\tau\left[\left(P(Z^{(n)}+Y^{(n)})+Q(Z^{(n)}\right)^*\cdot \left(P(Z^{(n)}+W^{(n)})+Q(Z^{(n)}\right)\right].
\end{align*}
\end{lemma}
\begin{proof}[Proof of Lemma]For all $i(1),\ldots,i(n)\in I$, we define a polynomial $P_{i(1),\ldots,i(n)}\in\C\langle x_i:i\in I \rangle$ by the following recursion formula: $P_{\emptyset}=1$ and 
$$P_{i(1),\ldots,i(n)}=x_{i(1)}\cdot P_{i(2),\ldots,i(n)}-s\sum_{k=2}^n\delta_{i(1)i(k)}q_{i(1)i(2)}\cdots q_{i(1)i(k-1)}\cdot P_{i(1),\ldots,\widehat{i(k)},\ldots, i(n)}$$
where the hat means that we omit the corresponding index. Since $\{P_{i(1),\ldots,i(n)}\}_{n\geq 0,i(1),\ldots,i(n)\in I}$ is a spanning set of $\C\langle x_i:i\in I \rangle$, it suffices to prove the theorem for those polynomials.

We have
$$P_{i(1),\ldots,i(n)}(\sqrt{s}X_i:i\in I)\Omega=\sqrt{s}e_{i(1)}\otimes \cdots \otimes \sqrt{s}e_{i(n)}.$$
Indeed, the definition of the polynomials $\{P_{i(1),\ldots,i(n)}\}_{n\geq 0,i(1),\ldots,i(n)\in I}$ is adjusted with the definition \eqref{defannihilation} of the annihilation operator $c^*_{i(1)}$ in such a way that, by a direct induction, for all $i(1),\ldots,i(n)\in I$, we have
\begin{align*}P_{i(1),\ldots,i(n)}(X_i:i\in I)\Omega=&\sqrt{s}(c_{i(1)}+c^*_{i(1)})\cdot \sqrt{s}e_{i(2)}\otimes \cdots \otimes \sqrt{s}e_{i(n)}\\
&-s\sum_{k=2}^n\delta_{i(1)i(k)}q_{i(1)i(2)}\cdots q_{i(1)i(k-1)}\cdot \sqrt{s}e_{i(2)}\otimes \cdots \otimes \widehat{e_{i(k)}} \otimes \cdots \otimes\sqrt{s}e_{i(n)}\\
=&\sqrt{s}c_{i(1)}\sqrt{s}e_{i(2)}\otimes \cdots \otimes \sqrt{s}e_{i(n)}\\
=&\sqrt{s}e_{i(1)}\otimes \cdots \otimes \sqrt{s}e_{i(n)}.
\end{align*}
Similarly, setting $h_i=\sqrt{s-t/2}e_{(i(n),1)}+i\sqrt{t/2}e_{(i(n),2)}+\sqrt{s}e_{(i(n),3)}$, it follows from the definition of the polynomials $\{P_{i(1),\ldots,i(n)}\}_{n\geq 0,i(1),\ldots,i(n)\in I}$ that
$$P_{i(1),\ldots,i(n)}(Z_i+Y_i:i\in I)\Omega=h_{i(1)}\otimes \cdots \otimes h_{i(n)}.$$
Indeed, for all $i(1),\ldots,i(n)\in I$, we have
\begin{align*}&(\sqrt{s-t/2}c_{(i(1),1)}^*+i\sqrt{s-t/2}c_{(i(1),2)}^*+\sqrt{t}c_{(i(2),3)}^*)\cdot h_{i(2)}\otimes \cdots \otimes h_{i(n)}\\
&=(s-t/2-t/2+t)\sum_{k=2}^n\delta_{i(1)i(k)}q_{i(1)i(2)}\cdots q_{i(1)i(k-1)}\cdot h_{i(2)}\otimes \cdots \otimes \widehat{h_{i(k)}} \otimes \cdots \otimes h_{i(n)}\\
&=s\sum_{k=2}^n\delta_{i(1)i(k)}q_{i(1)i(2)}\cdots q_{i(1)i(k-1)}\cdot h_{i(2)}\otimes \cdots \otimes \widehat{h_{i(k)}} \otimes \cdots \otimes h_{i(n)}
\end{align*}
which allows to write the induction step
\begin{align*}P_{i(1),\ldots,i(n)}(Z_i+Y_i:i\in I)\Omega=&(\sqrt{s-t/2}c_{(i(1),1)}+i\sqrt{s-t/2}c_{(i(1),2)}+\sqrt{t}c_{(i(2),3)})\cdot h_{i(2)}\otimes \cdots \otimes h_{i(n)}\\
&+(\sqrt{s-t/2}c_{(i(1),1)}^*+i\sqrt{s-t/2}c_{(i(1),2)}^*+\sqrt{t}c_{(i(2),3)}^*)\cdot h_{i(2)}\otimes \cdots \otimes h_{i(n)}\\
&-s\sum_{k=2}^n\delta_{i(1)i(k)}q_{i(1)i(2)}\cdots q_{i(1)i(k-1)}\cdot h_{i(2)}\otimes \cdots \otimes \widehat{h_{i(k)}} \otimes \cdots \otimes h_{i(n)}\\
=&(\sqrt{s-t/2}c_{(i(1),1)}+i\sqrt{s-t/2}c_{(i(1),2)}+\sqrt{t}c_{(i(2),3)})\cdot h_{i(2)}\otimes \cdots \otimes h_{i(n)}\\
=&h_{i(1)}\otimes \cdots \otimes h_{i(n)}.
\end{align*}
Setting $k_i=\sqrt{s-t/2}e_{(i(n),1)}+i\sqrt{t/2}e_{(i(n),2)}+\sqrt{s}e_{(i(n),4)}$, the same computation yields
$$P_{i(1),\ldots,i(n)}(Z_i+W_i:i\in I)\Omega=k_{i(1)}\otimes \cdots \otimes k_{i(n)}.$$
Finally, we have
\begin{align*}
&\tau\left[\left(P_{i(1),\ldots,i(n)}(Z_i+Y_i:i\in I)+Q(Z_i:i\in I)\right)^*\cdot \left(P_{i(1),\ldots,i(n)}(Z_i+W_i:i\in I)+Q(Z_i:i\in I)\right)\right]\\
&=\left\langle k_{i(1)}\otimes \cdots \otimes k_{i(n)}- Q(Z_i:i\in I)\Omega,h_{i(1)}\otimes \cdots \otimes h_{i(n)}- Q(Z_i:i\in I)\Omega\right\rangle_R.
\end{align*}
Because the $e_{(i,3)}$ occur only in $h_{i(1)}\otimes \cdots \otimes h_{i(n)}$, and the $e_{(i,4)}$ occur only in $k_{i(1)}\otimes \cdots \otimes k_{i(n)}$, they do not contribute to the scalar product, and we can replace
$h_{i(1)}\otimes \cdots \otimes h_{i(n)}$ and $k_{i(1)}\otimes \cdots \otimes k_{i(n)}$ by
$$(\sqrt{s-t/2}e_{(i(n),1)}+i\sqrt{t/2}e_{(i(n),2)})\otimes\cdots \otimes (\sqrt{s-t/2}e_{(i(n),1)}+i\sqrt{t/2}e_{(i(n),2)})=\delta_Q^{s,t}(\sqrt{s}e_{i(1)}\otimes \cdots \otimes \sqrt{s}e_{i(n)}),$$
which yields
\begin{align*}
&\tau\left[\left(P_{i(1),\ldots,i(n)}(Z_i+Y_i:i\in I)+Q(Z_i:i\in I)\right)^*\cdot \left(P_{i(1),\ldots,i(n)}(Z_i+W_i:i\in I)+Q(Z_i:i\in I)\right)\right]\\
&=\left\langle k_{i(1)}\otimes \cdots \otimes k_{i(n)}- Q(Z_i:i\in I)\Omega,h_{i(1)}\otimes \cdots \otimes h_{i(n)}- Q(Z_i:i\in I)\Omega\right\rangle_R\\
&=\left\langle \delta_Q^{(s,t)}(\sqrt{s}e_{i(1)}\otimes \cdots \otimes \sqrt{s}e_{i(n)})- Q(Z_i:i\in I)\Omega,\delta_Q^{(s,t)}(\sqrt{s}e_{i(1)}\otimes \cdots \otimes \sqrt{s}e_{i(n)})- Q(Z_i:i\in I)\Omega\right\rangle_R\\
&=\left\|\mathscr{S}_{Q}^{s,t}\left(P_{i(1),\ldots,i(n)}(X_i:i\in I)\right)-Q(Z_i:i\in I)\right\|_{\mathcal{H}L^2(\{Z_i\}_{i\in I},\tau)}.
\end{align*}
\end{proof}

Let us denote by $Q$ the polynomial $\mathscr{S}_q^{s,t}P$. Let $Y,W$ be two $(t,0)$-elliptic $q$-Gaussian random variables and $Z$ be a $(s-t/2,t/2)$-elliptic $q$-Gaussian random variable such that $Y,W$ and $Z$ are $q$-independent. Thanks to the discussion before Theorem~\ref{theoremfour}, we know that we can apply \cite[Theorem 1]{Speicher1992} in the case of $q_{ij}\in\{-1,+1\}$ (or \cite[Theorem 4]{Mlotkowski2004} in the case $q_{ij}\in\{0,+1\}$) which says that the mixed $q$-Gaussian random variables $Y^{(n)},W^{(n)},\Re Z^{(n)}$ and $\Im Z^{(n)}$ converge in noncommutative distribution to the $q$-Gaussian random variables $Y,W,\Re Z$ and $\Im Z$. In particular, we have the following convergence:
\begin{multline*}\lim_{n\to \infty}\tau\left[(P(Z^{(n)}+Y^{(n)})-Q(Z^{(n)}))^*(P(Z^{(n)}+W^{(n)})-Q(Z^{(n)}))\right]\\= \tau\left[(P(Z+Y)-Q(Z))^*(P(Z+W)-Q(Z))\right].\end{multline*}
From Corollary~\ref{CondExpProp} and Corollary~\ref{Condexpforone}, we know that $$Q(Z)=\mathscr{S}_q^{s,t}P(Z)=\tau[P(Z+Y)|Z]=\tau[P(Z+Y)|Z,W].$$
Thus the limit $\tau\left[(P(Z+Y)-Q(Z))^*(P(Z+W)-Q(Z))\right]$ of $\left\|\mathscr{S}_{Q}^{s,t}\left(P(X^{(n)})\right)-Q(Z^{(n)})\right\|_{\mathcal{H}L^2(Z_{\varepsilon},\tau)}$ vanishes:
\begin{align*}\tau\left[(P(Z+Y)-Q(Z))^*(P(Z+W)-Q(Z))\right]&=\tau\left[(P(Z+Y)-\tau[P(Z+Y)|Z,W])^*(P(Z+W)-Q(Z))\right]\\
&=\tau\left[(P(Z+Y)-P(Z+Y))^*(P(Z+W)-Q(Z))\right]\\
&=0.
\end{align*}
\end{proof}

\section*{Acknowledgements}

The first author was partially funded by the ERC Advanced Grant ``NCDFP'' held by Roland Speicher. The second author was funded by the same ERC Advanced Grant ``Non-commutative distributions in free probability'' (grant no. 339760). The second author would like to thank Roland Speicher for allowing his stay in Saarbr\"{u}cken, Germany so the authors had a chance to collaborate.

\bibliographystyle{acm}
\bibliography{SB_transform}

\end{document}